\documentclass[12pt, a4paper,reqno]{amsart}
\usepackage{amssymb,amsmath,latexsym,amsbsy}

\allowdisplaybreaks[2]
\begin{document}
\def\supp{\operatorname{supp}}
\def\tr{\operatorname{tr}}
\def\lt{{\operatorname{{lt}}}}
\def\id{\operatorname{{id}}}
\def\interior{\operatorname{Int}} 
\def\Ind{\operatorname{Ind}} 
\def\Res{\operatorname{Res}}
\def\Prim{\operatorname{Prim}}
\def\Ideal{\operatorname{Id}}
\def\Ideals{\operatorname{Id}}
\def\Aut{\operatorname{Aut}}
\newcommand{\Chi}{\raisebox{2pt}{\ensuremath{\chi}}}
\def\H{\mathcal{H}} 
\def\K{\mathcal{K}} 
\def\N{\mathcal{N}} 
\def\C{\mathbb{C}}
\def\T{\mathbb{T}}
\def\Z{\mathbb{Z}}
\def\R{\mathbb{R}}
\def\P{\mathbb{P}}
\def\NN{\mathbb{N}}
\newtheorem{thm}{Theorem}[section]
\newtheorem{cor}[thm]{Corollary}
\newtheorem{prop}[thm]{Proposition}
\newtheorem{lemma}[thm]{Lemma}
\theoremstyle{definition}
\newtheorem{defn}[thm]{Definition}
\newtheorem{remark}[thm]{Remark}
\newtheorem{example}[thm]{Example}
\newtheorem{examples}[thm]{Examples}
\numberwithin{equation}{section}
\title
[\boldmath Strength of convergence in non-free transformation groups]
{\boldmath Strength of convergence in non-free transformation groups}

\author[Archbold]{Robert Archbold}
\address{Institute of Mathematics
\\University of Aberdeen
\\Aberdeen AB24 3UE
\\Scotland
\\United Kingdom
}
\email{r.archbold@abdn.ac.uk}

\author[A. an Huef]{Astrid an Huef}
\address{Department of Mathematics and Statistics\\
University of Otago\\
Dunedin 9054\\
New Zealand}
\email{astrid@maths.otago.ac.nz}

\date{18 November 2011}

\begin{abstract} Let $(G, X)$ be a transformation group where the group $G$ does not necessarily act freely on the space $X$. We investigate the extent to which the action of $G$ may fail to be proper. Stability subgroups are used to define new notions of strength of convergence in the orbit space and of measure accumulation along orbits. By using the representation theory of the associated crossed product $C^*$-algebra, we show that these notions are equivalent under certain conditions.
\end{abstract}

\thanks{This research was supported by  grants from  the University of Otago, and by grant number 41019 from the London Mathematical Society.}

\maketitle

\section{Introduction}
Let $(G, X)$ be a second-countable, locally compact, Hausdorff
transformation group, so that the group $G$ acts continuously on the
space $X$. Thinking of the group action as time evolution on a state
space, an action is proper if states move far away from their
original position over long periods of time. We showed in \cite{AaH}
that the failure of properness in a free action of $G$ on $X$ can be
counted  by two methods, one topological and the other
measure-theoretic, and that they give the same answer. The
topological method is based on the notion of strength of convergence
from \cite[Definition~2.2]{AD} and is motivated by the behaviour of
sequences in the dual of a nilpotent Lie group \cite{Lud} and by
several previous considerations of the failure of properness
\cite{green1, aH, MW, W}. For example, a sequence $(x_n)_n$ in $X$
\emph{converges $2$-times in the orbit space to $z$ in $X$} if there
are two sequences $(t_n^{(1)})_n$ and $(t_n^{(2)})_n$ in $G$ such
that $(t_n^{(1)}\cdot x_n)_n$ and $(t_n^{(2)}\cdot x_n)_n$ both
converge to $x$, and $t_n^{(2)}(t_n^{(1)})^{-1}\to \infty$. The
measure-theoretic method involves the accumulation of Haar measure
as in \cite{aH}.  These methods are linked via the representation
theory of the crossed product $C^*$-algebra $C_0(X)\rtimes G$, and
in particular via the lower relative multiplicity number
$M_L(\pi,(\pi_n))$ associated to a particular convergent sequence
$\pi_n\to\pi$ in the spectrum of $C_0(X)\rtimes G$.

There are two natural ways to generalise these results. First, we can replace $X$ by a `non-commutative space', that is, replace $C_0(X)$ by a non-commutative $C^*$-algebra $A$ on which $G$ acts by automorphisms; we investigated this in \cite{AaH2} under the assumption that the induced action of $G$ on the primitive ideal space of $A$ is free. Second, we can retain the transformation group $(G, X)$ and relax the  assumption that the action of $G$ on $X$ should be free. In this paper we focus on non-free actions of $G$ on $X$.

When the  action is not free, the technical difficulties  of working with induced representations of the crossed product $C_0(X)\rtimes G$ increase substantially \cite{E94, E-memoir, ER,  green2, RW,  W2, W}. For this reason, and because we want to make use of the dual action on $C_0(X)\rtimes G$, we assume that the group $G$ is abelian in \S\S\ref{sec-lowerbounds}--\ref{main}.

In \S\ref{sec-prelim} we discuss our set-up which involves careful
choices of measures on the subgroups and quotients of $G$. In
\S\ref{sec-ktimesintro} we define $k$-times convergence in the
presence of stability  subgroups and show that measure accumulation
gives rise to sequences which converge $k$-times. In
\S\ref{sec-lowerbounds} we introduce induced representations of the
crossed product $C_0(X)\rtimes G$. We consider a sequence of induced
representations $(\pi_n)_n$ converging to an induced representation
$\pi$, and   establish sufficient conditions involving $k$-times
convergence which ensure that $M_L(\pi,(\pi_n))$ is bounded below.
In \S\ref{measure} we establish upper bounds on  $M_L(\pi,(\pi_n))$
arising from bounds on measure accumulation. In \S\ref{main}  we
combine our results to obtain our main theorem
(Theorem~\ref{thm-main}) which shows that, under certain conditions,
strength of convergence and measure accumulation in the
transformation group $(G,X)$ are equivalent, being linked by the
representation theory of $C_0(X)\rtimes G$.

\subsection*{Preliminaries} Let $A$ be a $C^*$-algebra, $\hat A$ its spectrum
and $\pi\in\hat A$ an irreducible representation. Upper and lower multiplicities
$M_U(\pi)$ and $M_L(\pi)$, and upper and lower multiplicities $M_U(\pi,(\pi_n))$
and $M_L(\pi,(\pi_n))$ relative to a net $(\pi_n)$ in $\hat A$ were first defined
in \cite{A} and \cite{AS}, respectively. We refer the reader to \cite[\S2]{AaH} for a convenient summary of what is needed here. We set $\P=\mathbb N\setminus\{0\}$.

 \section{The set-up: Choices of measures on the subgroups of $G$}\label{sec-prelim}
Let $G$ be a locally compact group with left Haar measure $\mu$.
Let $\Sigma$ be the family of all closed subgroups of $G$.
We endow $\Sigma$ with the \emph{Fell topology} from \cite{fell-topology}.
A basis for this topology is the family of  sets
\[
U(C,F)=\lbrace H\in\Sigma: C\cap H=\emptyset \text{\ and\ }
H\cap A\neq\emptyset \text{\ for\ each\ }A\in F\rbrace,
\]
where  $C$ is a  compact subset of $G$ and $F$ a finite family of non-empty open subsets of $G$.
Then $\Sigma$ is a compact Hausdorff space by  \cite[Theorem 1 and Remark IV]{fell-topology}.
We will frequently use that $H_{\lambda}\to H$ in $\Sigma$ if and only if
\begin{enumerate}
\item if $h\in H$ then there exist a subnet  $(H_{{\lambda(\mu)}})$ and  $h_{\mu}\in H_{{\lambda(\mu)}}$ such that $h_{\mu}\to h$, and
\item if  $h_\lambda\in H_{\lambda}$ and $h_\lambda\to h$ then $h\in H$
\end{enumerate}
(see, for example, \cite[Lemma~H2]{tfb^2}).

Fix a  function $f_0\in C_c(G)$ with $f_0(e)=1$ and $0\leq f\leq 1$.
For each $H\in\Sigma$ we choose the left Haar measure $\alpha_H$ on $H$ satisfying
\[\int_H f_0(t) d\alpha_H(t)=1.\]
Such a choice of measures is called a \emph{continuous choice of Haar measures} on the closed
subgroups of $G,$ and has the property that
\[H\mapsto \int_H f(t) \ d\alpha_H(t)\]
is a continuous function on $\Sigma$ for any $f\in C_c(G)$  (see \cite[p. 908]{glimm-families} or \cite[Lemma~H.8]{tfb^2}). We write $\Delta_H$ for the
modular function associated with the measure $\alpha_H$.

Now let $(G,X)$ be a transformation group. The \emph{stability subgroup} at $x\in X$ is $S_x:=\{s\in G:s\cdot x=x\}$.
We write $\alpha_x$ for $\alpha_{S_x}$, $q_x:G\to G/S_x$ for the quotient map, and $\dot s=sS_x=q_x(s)$ for $s\in G$. We also define $\phi_x:G\to X$ by $\phi_x(s)=s\cdot x$.

If $H$ is a normal subgroup of $G$ then there exists a unique right-invariant Haar measure $\nu_H$
on $G/H$ such that for all
$f\in C_c(G)$,
\begin{equation}\label{quotientmeasure1}\int_G f(s)\ d\mu(s)=\int_{G/H}\int_{H} f(st)
\ d\alpha_H(t) \ d\nu_H(\dot s)
\end{equation}
(see, for example, \cite[Appendix~C]{tfb}).  If $H=S_x$ then we write $\nu_x$ for $\nu_{S_x}$.
We claim that it follows that if $\chi_E$ is the characteristic function of a measurable subset $E$ of $G$, then
\begin{equation}\label{quotientmeasure2}\int_G \chi_E(s)f(s) \ d\mu(s)=\int_{G/H}\int_{H} \chi_E(st)f(st)
\ d\alpha_H(t) \ d\nu_H(\dot s).
\end{equation}
Indeed, since $\supp f$ is compact, we may assume that $E$ has finite measure so that by Lusin's Theorem $\chi_E$ is the pointwise limit (almost everywhere) of a sequence of continuous functions $g_n: G \to [0,1]$. Then the claim follows by repeated applications of the Dominated Convergence Theorem.

Typically we will be interested in sequences $(x_n)_n$ in $X$ such that $S_{x_n}\to S_z$ for some $z\in X$ as $n\to \infty$, and that $S_{x_n}$ and $S_z$ are normal in $G$; we then assume that the quotient measures $\nu_{x_n}$ and $\nu_z$ have been chosen to satisfy \eqref{quotientmeasure1}.  When $x_n\to z$ in $X$, the assumption that $S_{x_n}\to S_z$ is, of course, weaker than the assumption that the stability subgroups vary continuously over the whole space $X$.

\section{$k$-times convergence with stability}\label{sec-ktimesintro}

The following definition of $k$-times convergence immediately reduces to the one used in \cite{AD, AaH} when the action of $G$ on $X$ is free.

\begin{defn}\label{defn_k-times} Let $(G,X)$ be a transformation group.
A sequence $(x_n)_{n\geq 1}$ in $X$ is
said to \emph{converge $k$-times in $X/G$ to $z\in X$} if there
exist $k$ sequences $(t_n^{(1)})_n,(t_n^{(2)})_n,\cdots
,(t_n^{(k)})_n$ in $G$ such that
\begin{enumerate}
\item $t_n^{(i)}\cdot x_n\to z$ as $n\to\infty$ for $1\leq i\leq k$, and
\item if $1\leq i<j\leq k$ then $t_n^{(j)}(t_n^{(i)})^{-1}S_{x_n}\to\infty$ as
$n\to\infty$ (that is, for every compact subset $K$ of $G$,  $t^{(j)}_n(t^{(i)}_n)^{-1}S_{x_n}$ is eventually disjoint from $K$).
\end{enumerate}
\end{defn}

As in the free case,  $k$-times converging sequences in $X/G$ arise
from measure accumulation. We show this in
Proposition~\ref{prop-tsoc2} below, and the next three lemmas lead
towards this. In order to describe measure accumulation for a point
$z\in X$, we require that $z$ has a base of neighbourhoods $V$ such
that $q_z(\phi_z^{-1}(V))$ has finite Haar measure in $G/S_z$, and by Lemma~\ref{lem-remark}
this is equivalent to the orbit $G\cdot z$
being locally closed in $X$.  Lemmas~\ref{lem-pointwise} and \ref{lem-supmeasures2} are technical ones addressing the following issue: if $W$ is a compact neighbourhood of $G$ and $S_{x_n}\to S_z$ in $\Sigma$, then we want to compare the measures $\nu_{x_n}(q_{x_n}(WS_{x_n}))$ and $\nu_{z}(q_{z}(WS_{z}))$ for large $n$. But even though $WS_{x_n}\to WS_z$  in the Fell topology on the closed subsets of $X$,  it is conceivable that  $\chi_{WS_{x_n}}$ may not converge pointwise almost everywhere to $\chi_{WS_z}$.

\begin{lemma}\label{lem-remark}
Let $(G,X)$ be a second-countable transformation group. Let $z\in X$
and suppose that the stability subgroup $S_z$ is normal in $G$. Then the following are
equivalent:
\begin{enumerate}
\item\label{lem-remark-1} the orbit $G\cdot z$ is not locally closed in $X$;
\item\label{lem-remark-2} for every $k\in\P$, the sequence $z,z,z,\dots$ converges
$k$-times in $X/G$ to $z$;
\item\label{lem-remark-3} for every open neighbourhood $V$ of $z$,
$\nu_z(q_z(\phi_z^{-1}(V)))=\infty$;
\item\label{lem-remark-4} for every open neighbourhood $V$ of $z$,
$q_z(\phi_z^{-1}(V))$ is not relatively compact in $G/S_z$.
\end{enumerate}
\end{lemma}
\begin{proof}
Let $(V_n)_{n\geq 1}$ be a decreasing sequence of basic open
neighbourhoods of $z$ in $X$ and let $(K_n)_{n\geq 1}$ be an
increasing sequence of compact subsets of $G$ such that
$G=\cup_{n\geq 1}\interior(K_n)$.

\eqref{lem-remark-1} $\Longrightarrow$ \eqref{lem-remark-2}. Suppose that $G\cdot z$ is not locally
closed. Then $W\cap (\overline{G\cdot z}\setminus G\cdot
z)\neq\emptyset$ for every neighbourhood $W$ of $z$. Let $k\geq1$.
We will construct the required $k$ sequences $(t_n^{(i)})_{n\geq 1}$ in $G$ by induction.

Let $n\in\P$. We construct $t_n^{(i)}$ as follows.  Let
$t_n^{(1)}=e$. Since $G\cdot z$ is not locally closed there exists
$y\in V_n\cap (\overline{G\cdot z}\setminus G\cdot z)$. Since $y$ is
in the closure of $G\cdot z$ and $V_n$ is open, a straightforward compactness argument shows that given any compact
subset $K$ of $G$ there exists $t_K\in G\setminus KS_z$ such that
$t_K\cdot z\in V_n$.  So there exists $t_n^{(2)}\in G\setminus
K_nS_z=G\setminus K_nt_n^{(1)}S_z$ such that $t_n^{(2)}\cdot z\in V_n$. Proceeding inductively
we obtain $t_n^{(2)},t_n^{(3)} \dots, t_n^{(k)}$ such that
\[t_n^{(j)}\cdot z\in V_n\quad\text{and}\quad t_n^{(j)}\in G\setminus  \left(\cup_{i=1}^{j-1} K_nt_n^{(i)}S_z\right)
\]
for $2\leq j\leq k$.

Since $(V_n)_{n\geq 1}$ is  a decreasing sequence of basic open
neighbourhoods of $z$, $t_n^{(j)}\cdot z\to z$  as $n\to\infty$ for $1\leq j\leq k$.
 By way of contradiction, suppose that for some $i<j$ there exists a compact subset $K$ of $G$ such that $t_n^{(j)}(t_n^{(i)})^{-1}S_z$ meets $K$ frequently. Since $(K_n)_{n\geq 1}$ is an
increasing sequence of compact subsets of $G$ such that
$G=\cup_{n\geq 1}\interior(K_n)$ there exists $N$ such that $K\subset K_N$.  Then there exsist $n_0\geq N$  such that  $t_{n_0}^{(j)}(t_{n_0}^{(i)})^{-1}S_z$ meets $K_{n_0}$.  But  this implies that  $t_{n_0}^{(j)}\in K_{n_0}t_{n_0}^{(i)}S_z$ because $S_z$ is normal, contradicting the construction of the sequence $(t_n^{(j)})$. Thus  $z,z,z,\dots$ converges $k$-times in $X/G$ to $z$.

\eqref{lem-remark-2} $\Longrightarrow$ \eqref{lem-remark-3}. Suppose that \eqref{lem-remark-2}  holds. Let $V$ be an
open neighbourhood of $z$ and $M>0$.  By the continuity of the
action on the locally compact Hausdorff space $X$, there exists an
open neighbourhood $U$ of $z$ and a compact neighbourhood $K$ of $e$
in $G$ such that $K\cdot U\subset V$. Then $q_z(K)$
is a compact neighbourhood of the identity in $G/S_z$ and we may
choose $k\in\P$ such that $k\nu_z(q_z(K))>M$. By \eqref{lem-remark-2} there exist $k$
sequences $(t_n^{(i)})_{n\geq 1}$
 such that $t_n^{(i)}\cdot z\to z$ as $n\to\infty$ for each $1\leq i\leq k$,
 and
\[ t_n^{(j)}(t_n^{(i)})^{-1}S_z\to\infty\text{\ as\ }
n\to\infty\quad (1\leq i<j\leq k).
\]
Hence there exists $n_0$ such that  $t_{n_0}^{(i)}\cdot z\in U$ for
$1\leq i\leq k$ and $t_{n_0}^{(j)}(t_{n_0}^{(i)})^{-1}S_z\subset
G\setminus(K^{-1}K)$ for $1\leq i<j\leq k$. Then
$Kt_{n_0}^{(i)}\cdot z\subset K\cdot U\subset V$  and hence $q_z(Kt_{n_0}^{(i)})\subset q_z(\phi_z^{-1}(V))$ for $1\leq i\leq k$. Also
$q_z(Kt_{n_0}^{(i)})\cap q_z( Kt_{n_0}^{(j)})=\emptyset$ because $Kt_{n_0}^{(i)}S_z\cap Kt_{n_0}^{(j)}S_z=\emptyset$ unless
$i=j$. Since the measure $\nu_z$ is right-invariant, $\nu_z(q_z(\phi_z^{-1}(V)))\geq k\nu_z(q_z(K))>M$.
Since $M$ was arbitrary, \eqref{lem-remark-3} follows.

\eqref{lem-remark-3} $\Longrightarrow$ \eqref{lem-remark-4}. Compact subsets have finite Haar measure,
so this is immediate.

\eqref{lem-remark-4} $\Longrightarrow$ \eqref{lem-remark-1}. Suppose that $G\cdot z$ is locally closed
 in $X$. Then $G\cdot z$ is a relatively open subset of the locally
 compact space $\overline{G\cdot z}$ and hence $G\cdot z$ is locally
 compact.  Thus $(G,G\cdot z)$ is a second-countable,
 locally compact, Hausdorff transformation group.  In particular, it
 follows from \cite[Theorem~1]{Gli} that the map $\psi_z:G/S_z\to
 G\cdot z:sS_z\mapsto s\cdot z$ is a homeomorphism.
Let $U$ be an open subset  of $X$ suc
h that $U\cap\overline{G\cdot
z}=G\cdot z$. Let $N$ be a compact neighbourhood of $z$ in $X$ such
that $N\subset U$. Then $N\cap \overline{G\cdot z}=N\cap G\cdot z$
is a compact subset of $G\cdot z$. Hence $\psi_z^{-1}(N)$ is compact
in $G/S_z$. Now $\psi_z^{-1}(N)=
q_z(\phi_z^{-1}(N))\supset q_z(\phi_z^{-1}({\rm Int}N))$; this contradicts \eqref{lem-remark-4} with $V=\interior N$.
\end{proof}

Lemma~\ref{lem-pointwise} is used in Lemma~\ref{lem-supmeasures2} and in Theorem~\ref{thm-lowermult} below.

\begin{lemma}\label{lem-pointwise}
Let $(G,X)$ be a second-countable transformation group. Let $z\in X$
and let  $(x_n)_n$ be a sequence in $X$ such that $S_{x_n}\to S_z$ in $\Sigma$.
Suppose that  $W$ is a compact subset of $G$. Then
\[
\Chi_{WS_{x_n}}(r)\to \Chi_{WS_z}(r)
\]
for $r\in(G\setminus WS_z)\cup (\interior W)S_z$. If
$\mu\big(WS_z\setminus(\interior W)S_z\big)=0$, then
$
\Chi_{WS_{x_n}}\to \Chi_{WS_z}
$
almost everywhere.
\end{lemma}
\begin{proof}  The second  statement follows immediately from the first.

Since $W$ is compact both $WS_{x_n}$ and  $WS_z$ are closed. We
claim  that $WS_{x_n}\to WS_z$ in the Fell topology on the closed
subsets of $X$. First let $s\in WS_z$, say $s=wt$ where $w\in W$ and
$t\in S_z$. There exist a subsequence $(x_{n_i})$ and  $t_{n_i}\in
S_{x_{n_i}}$ such that $t_{n_i}\to t$. So $wt_{n_i}\to wt=s$ and
$wt_{n_i}\in WS_{n_i}$ as required. Second, consider  $(s_n)$ with
$s_n\in WS_{x_n}$, say $s_n=w_nt_n$, and suppose that $s_n\to s$. By
compactness there exists a subsequence $(w_{n_i})$ such that
$w_{n_i}\to w\in W$. Then $t_{n_i}=w_{n_i}^{-1}s_{n_i}\to w^{-1}s$,
and hence $w^{-1}s\in S_z$. Thus $s\in wS_z\subset WS_z$ as required, and
$WS_{x_n}\to WS_z$ in the Fell topology as claimed.

It follows from $WS_{x_n}\to WS_z$ that $\Chi_{WS_{x_n}}(s)\to \Chi_{WS_z}(s)=0$ for $s\in G\setminus WS_z$. For if not then there exists a subsequence $x_{n_i}$ such that $s\in WS_{x_{n_i}}$. Set $s_{n_i}=s$ and note $s_{n_i}\to s$.  But then $s\in WS_z$ since $WS_{x_n}\to WS_z$, a contradiction.

Now let $s\in (\interior W)S_z$.  We first claim that there exists a
subsequence $(x_{n_i})$ such that $\Chi_{WS_{x_{n_i}}}(s)\to
\Chi_{WS_z}(s)=1$. To see this, write
$s=wt$ where $w\in \interior W$ and $t\in S_z$. Then there exist
$t_{n_i}\in S_{x_{n_i}}$ such that $t_{n_i}\to t$.  Choose a
symmetric neighbourhood $U$ of $e$ in $G$ such that $wU\subseteq
\interior W$.  Then $t_{n_i}\in Ut$ eventually, so eventually there
exist $u_i\in U$ such that $t_{n_i}=u_{i}t$. Thus
$s=wt=wu_{i}^{-1}t_{n_i}\in wUS_{n_i}\subseteq(\interior
W)S_{x_{n_i}}$. So $\Chi_{WS_{x_{n_i}}}(s)=1$ eventually, as
claimed.

Now suppose $s\in (\interior W)S_z$ and that
$\Chi_{WS_{x_n}}(s)\not\to \Chi_{WS_z}(s)=1$. Then there exists a subsequence $(x_{n(j)})$  such that $s\notin WS_{x_{n(j)}}$ for all $j$. But applying the argument of the preceding paragraph we get a further subsequence $(x_{n(j(i))})$ such that $s\in WS_{x_{n(j(i))}}$ eventually, a contradiction. Thus $\Chi_{WS_{x_n}}(s)\to \Chi_{WS_z}(s)=1$ for $s\in (\interior W)S_z$.
\end{proof}

\begin{lemma}\label{lem-supmeasures2}
Let $(G, X)$ be a second-countable transformation group.  Let $z\in X$, and let $(x_n)_{n\geq 1}$ be a sequence in $X$ such that $S_{x_n}\to S_z$  in $\Sigma$,  and  $S_z$ and all the $S_{x_n}$  are normal in $G$.  Let $W$ be a compact subset in $G$. Then $\limsup_n\nu_{x_n}(q_{x_n}(W))\leq  \nu_z(q_z(W))$.
\end{lemma}

\begin{proof}
Since $S_{x_n}\to S_z$, $\chi_{WS_{x_n}}(r)\to \chi_{WS_z}(r)$ for $r\in G\setminus WS_z$ by Lemma~\ref{lem-pointwise}.  By \cite[Proposition~2.18]{W} or \cite[Proposition~H.17]{tfb^2} there exists a ``cut-down approximate Bruhat cross section''  $b\in C_c(G\times \Sigma)$: thus  $b\geq 0$ and, for all $H\in\Sigma$,  $\int_{H} b(rt,H)\, d\alpha_H(t)=1$ for $r\in WH$. Now
\begin{align}
\nu_{x_n}(q_{x_n}(W))&=\int_{G/S_{x_n}}\chi_{q_{x_n}(W)}(\dot r)\, d\nu_{x_n}(\dot r)\notag\\
&=\int_{G/S_{x_n}}\chi_{q_{x_n}(W)}(\dot r) \int_{S_{x_n}} b(rt,S_{x_n})\, d\alpha_{x_n}(t)   \, d\nu_{x_n}(\dot r)\notag\\
&=\int_{G/S_{x_n}} \int_{S_{x_n}}\chi_{WS_{x_n}}(rt) b(rt,S_{x_n})\, d\alpha_{x_n}(t)   \, d\nu_{x_n}(\dot r)\notag\\
&=\int_G\chi_{WS_{x_n}}(r) b(r,S_{x_n})\, d\mu(r)\quad \text{(using \eqref{quotientmeasure2})}\label{samecalc}\\
&\leq \int_{r\in WS_z}b(r,S_{x_n})\, d\mu(r)+
\int_{r\in G\setminus WS_z}\chi_{WS_{x_n}}(r) b(r,S_{x_n})\, d\mu(r).\notag
\end{align}
Let $K$ be the image of $\supp b$ under the coordinate projection $G\times X\to G$.  Since $S_{x_n}\to S_z$, $b$ is continuous,  and $\chi_{WS_{x_n}}(r)\to 0$ for $r\in G\setminus WS_z$, we may apply the Dominating Convergence Theorem with dominating functions $\|b\|_\infty\chi_{(WS_z)\cap K}$ and $\|b\|_\infty\chi_{(G\setminus WS_z)\cap K}$ to show that the sum of integrals converges to
\begin{align*}
\int_{r\in WS_z}b(r,S_z)\, d\mu(r)+0
&= \int_G\chi_{WS_{z}}(r)b(r,S_z)\, d\mu(r).
\end{align*}
But $\int_G\chi_{WS_{z}}(r)b(r,S_z)\, d\mu(r)$ equals $\nu_{z}(q_{z}(W))$ by the calculation above ending at \eqref{samecalc}. Thus
\[
\nu_{x_n}(q_{x_n}(W))\leq \int_{r\in WS_z}b(r,S_{x_n})\, d\mu(r)+
\int_{r\in G\setminus WS_z}\chi_{WS_{x_n}}(r) b(r,S_{x_n})\, d\mu(r)\to \nu_{z}(q_{z}(W)),
\]
and the lemma follows.
\end{proof}

We can now extend \cite[Proposition~4.1]{AaH} to the non-free case.

\begin{prop}\label{prop-tsoc2}
Let $(G, X)$ be a second-countable transformation group. Let $z\in X$ with $G\cdot z$ locally closed in $X$ and $S_z$ compact. Assume that $(x_n)_{n\geq 1}$ is a sequence in $X$ such that $S_{x_n}\to S_z$ in $\Sigma$, that $S_{x_n}$ and $S_z$ are normal in $G$, and that $G\cdot z$ is the unique limit  of $(G\cdot x_n)_n$ in $X/G$.  Let $k\in\P$, and suppose that there exists a basic sequence $(W_m)_{m\geq 1}$ of compact neighbourhoods of $z$ with $W_{m+1}\subset W_m$, such that, for each $m$,
\[
\liminf_n\nu_{x_n}(q_{x_n}(\phi_{x_n}^{-1}(W_m)))>(k-1)\nu_z(q_z(\phi_z^{-1}(W_m))).
\]
Then  $(x_n)$  converges $k$-times in $X/G$ to $z$.
\end{prop}

\begin{proof}
Since $G\cdot z$ is locally closed in $X$, there is an open subset
$U$ of $X$ such that $U\cap\overline{G\cdot z}=G\cdot z$. We may
assume, without loss of generality, that $W_m\subset U$ for all
$m\geq1$. It then follows from \cite[Theorem~1]{Gli}, as in the
proof of Lemma~\ref{lem-remark}, that $q_z(\phi_z^{-1}(W_m))$ is
compact in $G/S_z$ for all $m\geq1$.  Let $(K_m)_{m\geq 1}$ be an
increasing sequence of compact subsets of $G$ such that
$G=\cup_{m\geq 1}\interior(K_m)$.

Let $m\geq 1$.  It follows from the regularity of $\nu_z$ that there
exists an open neighbourhood $V_m$ of $q_z(\phi_z^{-1}(W_m))$ and
$\epsilon_m>0$ such that
\begin{equation*}
\liminf_n\nu_{x_n}(q_{x_n}((\phi_{x_n}^{-1}(W_m))))>(k-1)\nu_z(V_m) +\epsilon_m.
\end{equation*}
Since  $q_z(\phi_z^{-1}(W_m))$ is compact and $G/S_z$ is locally
compact, there exists an open precompact neighbourhood $A_m$ of
$q_z(\phi_z^{-1}(W_m))$  such that $\overline{A_m}\subset V_m$. We
have
\begin{equation}\label{eq-reg-cl}
\liminf_n\nu_{x_n}(q_{x_n}((\phi_{x_n}^{-1}(W_m))))>(k-1)\nu_z(\overline{A_m})
+\epsilon_m.
\end{equation}
Set $U_m:=q_z^{-1}(A_m)$, so that $U_m$ is an open neighbourhood of $\phi_z^{-1}(W_m)$ and is precompact since both $\overline{A_m}$ and $S_z$ are compact.

We will  construct, by induction,  a strictly increasing sequence of
positive integers $(n_m)_{m\geq 1}$ such that, for all $n\geq
n_m$,
\begin{align}
\nu_{x_n}(q_{x_n}(K_msS_{x_n}\cap\phi_{x_n}^{-1}(W_m)))&\leq\nu_z(\overline{A_m})+\epsilon_m/k\quad\text{for all\
}s\in\phi_{x_n}^{-1}(W_m),
\text{\ and}\label{eq-excised}\\
\nu_{x_n}(q_{x_n}(\phi_{x_n}^{-1}(W_m)))&>(k-1)\nu_z(\overline{A_m})+\epsilon_m.
\label{eq-measure}
\end{align}
We construct $n_1$ by applying \cite[Lemma~3.2]{AaH} to
$K_1, W_1$ and $U_1$ to obtain $n_1$ such that for every $n\geq n_1$ and every $s\in\phi_{x_n}^{-1}(W_1)$ there exists $r\in \phi_z^{-1}(W_1)$ such that
$
K_1s\cap\phi_{x_n}^{-1}(W_1)\subset U_1r^{-1}s.
$
Then
$
K_1s S_{x_n}\cap\phi_{x_n}^{-1}(W_1)\subset U_1r^{-1}s S_{x_n}
$
 and hence
\[\nu_{x_n}(q_{x_n}(K_1sS_{x_n}\cap\phi_{x_n}^{-1}(W_1)))\leq \nu_{x_n}(q_{x_n}(U_1r^{-1}sS_{x_n}))=\nu_{x_n}(q_{x_n}(U_1))\]
because $S_{x_n}$ is normal and the measures $\nu_{x_n}$ have been chosen to be right invariant.
Since $S_{x_n}\to S_z$ ,  $\limsup \nu_{x_n}(q_{x_n}(\overline{U_1}))\leq \nu_z(q_z(\overline{U_1}))\leq \nu_z(\overline{A_1})$ by  Lemma~\ref{lem-supmeasures2}. So  by increasing $n_1$ if necessary, \eqref{eq-excised} holds for $m=1$.
If necessary, we can increase $n_1$ again to ensure that \eqref{eq-measure}
holds by using \eqref{eq-reg-cl} with $m=1$.

Assuming that we have constructed $n_1<n_2<\dots<n_{m-1}$,
we apply \cite[Lemma~3.2]{AaH} to $K_m, W_m$ and $U_m$
to obtain $n_m>n_{m-1}$ such that for $n>n_m$ and  every $s\in\phi_{x_n}^{-1}(W_m)$
 there exists  $r\in \phi_z^{-1}(W_1)$ such that
$
K_ms\cap\phi_{x_n}^{-1}(W_m)\subset U_mr^{-1}s.
$
Then
$
K_ms S_{x_n}\cap\phi_{x_n}^{-1}(W_m)\subset U_mr^{-1}s S_{x_n}.
$
Applying Lemma~~\ref{lem-supmeasures2} to $\overline{U_m}$, we have
\[\nu_{x_n}(q_{x_n}(K_msS_{x_n}\cap\phi_{x_n}^{-1}(W_m)))
\leq \nu_{x_n}(q_{x_n}(U_m))\leq \nu_{x_n}(q_{x_n}(\overline{U_m}))\leq \nu_z(q_z(\overline{U_m})) +\epsilon_m/k\]
eventually.
Since $q_z(\overline{U_m})\subset A_m$ we can  increase $n_m$ twice if necessary to obtain first  \eqref{eq-excised} and then \eqref{eq-measure}.

If $n_1>1$ then, for $1\leq n<n_1$, we set $t^{(i)}_n=e$ for $1\leq i\leq k$.
For each $n\geq n_1$ there is a unique $m$ such that $n_m\leq n< n_{m+1}$.
Choose $t^{(1)}_n\in\phi_{x_n}^{-1}(W_m)$.  Using \eqref{eq-excised} and \eqref{eq-measure}
\begin{align}
\nu_{x_n}\big(q_{x_n}\big(\phi_{x_n}^{-1}(W_m)\setminus K_mt^{(1)}_nS_{x_n}\big)\big)
&\geq \nu_{x_n}\big(q_{x_n}(\phi_{x_n}^{-1}(W_m))\setminus
q_{x_n}(\phi_{x_n}^{-1}(W_m)\cap K_mt^{(1)}_nS_{x_n})\big)\notag\\
&>(k-2)\nu_z(\overline{A_m})+(k-1)\epsilon_m/k.\label{eq-new}
\end{align}
So if $k\geq 2$ we may choose $t_n^{(2)}$  such that $t_n^{(2)}\in \phi_{x_n}^{-1}(W_m)\setminus K_mt^{(1)}_nS_{x_n}$.

Next, using the formal relation $q_{x_n}(A\setminus (B\cup
C))\supset q_{x_n}(A\setminus B)\setminus q_{x_n}(A\cap C)$, and
\eqref{eq-excised} and \eqref{eq-new}
\[
\nu_{x_n}(q_{x_n}(\phi_{x_n}^{-1}(W_m)\setminus(
K_mt^{(1)}_nS_{x_n}\cup K_mt^{(2)}_nS_{x_n})))>
(k-3)\nu_z(\overline{A_m})+(k-2)\epsilon_m/k.
\]
So if $k\geq 3$ we may choose  $t_n^{(3)}\in \phi_{x_n}^{-1}(W_m)\setminus (K_mt^{(1)}_nS_{x_n}\cup K_mt^{(2)}_nS_{x_n})$.
Continuing in this way, we obtain  $t_n^{(1)},\dots, t_n^{(k)}\in \phi_{x_n}^{-1}(W_m)$ such that, for $1\leq j\leq k$
\begin{equation*}
t_n^{(j)}\in \phi_{x_n}^{-1}(W_m)\setminus (\cup_{i=1}^{j-1}K_mt^{(i)}_nS_{x_n}).
\end{equation*}
Note that for $n_m\leq n<n_{m+1}$ we have
\[
t_n^{(i)}\cdot x_n\in W_m\text{\ for\ }1\leq i\leq k\text{\ and}\quad
t_n^{(j)}\notin K_mt_n^{(i)}S_{x_n}\text{\ for\ }1\leq i<j\leq k.
\]

Now we will show that $x_n$ converges $k$-times in $X/G$ to $z$.
To see that $t_n^{(i)}\cdot x_n\to z$ as $n\to\infty$ for $1\leq i\leq k$, fix $i$ and let $V$ be a neighbourhood of $z$.
There exists $m_0$ such that $W_m\subset V$ for all $m\geq m_0$.
For each $n\geq n_{m_0}$ there exists   $m\geq m_0$ such that $n_m\leq n<n_{m+1}$,
and thus $t_n^{(i)}\cdot x_n\in W_m\subset V$. Thus  $t_n^{(i)}\cdot x_n\to z$ as $n\to\infty$ for $1\leq i\leq k$ as claimed.

Finally, fix a compact set $K$ in $G$. There exists $m_0$ such that $K\subset K_m$ for all $m\geq m_0$. Then for $m\geq m_0$, $n_m\leq n<n_{m+1}$ and $1\leq i<j\leq k$ we have $t_n^{(j)}\notin K_mt_n^{(i)}S_{x_n}=K_mS_{x_n}t_n^{(i)}$ and hence $t_n^{(j)} (t_n^{(i)})^{-1}S_{x_n}\subset G\setminus K$.   So for $n\geq n_{m_0}$  and $1\leq i<j\leq k$, $t_n^{(j)}(t_n^{(i)})^{-1}S_{x_n}\subset G\setminus K$ as required.
\end{proof}

\section{$k$-times convergence and  lower bounds on multiplicity}\label{sec-lowerbounds}

Throughout this section $G$ is assumed to be abelian. Let $k\in\P$. Here we introduce induced representations of the crossed product $C_0(X)\rtimes G$, and consider a sequence $(\pi_n)_n$ of induced representations  converging to an induced representation $\pi$. We  establish sufficient conditions which ensure that $M_L(\pi,(\pi_n))\geq k$. The dual action of the character group $\hat G$ on $C_0(X)\rtimes G$ plays a major role in our approach.

We start with some background on induced representations.
If $\tau\in\hat G$ then ``the representation $\Ind_{x,S_x}^G(\tau|)$ of $C_0(X)\rtimes G$ induced from $\tau|$ on $S_x$'' is irreducible by \cite[Lemma~4.14]{W2}.  By \cite[Proposition~4.2]{W2} $\Ind_{x,S_x}^G(\tau|)$ is unitarily equivalent to $\Ind(x,\tau):=M_x\rtimes V_\tau$ on $L^2(G/S_x,\nu_x)$,
where
\[(M_x(f)\xi)(\dot s) = f(s\cdot x)\xi(\dot s)\quad\text{and}
\quad(V_\tau(t)\xi)(\dot s)=\tau(t)\xi(t^{-1}\cdot \dot s)\]
for  $f\in C_0(X)$ and $\xi\in L^2(G/S_x)$.

Let $x, y\in X$ and $\tau,\sigma\in\hat G$.  Write $(x,\tau)\!\sim (y,\sigma)$ if and only if $\overline{G\cdot x}=\overline{G\cdot y}$ and $\tau|_{S_x}=\sigma|_{S_x}$. Then $\sim$ is an equivalence relation. Since $G$ is abelian, Theorem~5.3 of \cite{W2} implies that
the map $(x,\tau)\mapsto \ker \Ind(x,\tau)$ induces a homeomorphism of $(X\times \hat G)/\!\sim$ onto the primitive ideal space of $C_0(X)\rtimes G$.  The proof of  \cite[Theorem~5.3]{W2} also shows that the quotient map $X\times\hat G\to (X\times \hat G)/\!\sim$ is open.

Let $\hat\alpha$ be the dual action of $\hat G$ on $C_0(X)\rtimes G$, that is,
\[\hat\alpha_\tau(b)(s)=\tau(s)b(s)\ \text{for\ } b\in C_c(G,C_0(X))\ \text{and}\ \tau\in\hat G.\]
An easy calculation  shows that $\Ind(x,\tau)=\Ind(x,1)\circ\hat\alpha_\tau$  \cite[Lemma~5.5]{aH}.
The next two lemmas will be needed for the proof of
Proposition~\ref{prop-constant} where we will show that if $\Ind(x_n,\tau_n)\to \Ind(z,\tau)$ then $M_L (\Ind(z,\tau)  , (\Ind( x_n,\tau_n) ) )$ does not depend on $(\tau_n)$ or $\tau$.

\begin{lemma}\label{lem-sub} Let $(G,X)$ be a second-countable transformation
 group with $G$
abelian. Suppose that  a net
$(\Ind(x_\lambda,\tau_\lambda))_{\lambda\in\Lambda}$ converges to $
\Ind(z,\tau)$ in the spectrum of $C_0(X)\rtimes G$. Then there
exists a subnet
$(\Ind(x_{\lambda(\gamma)},\tau_{\lambda(\gamma)}))_{\gamma\in
\Gamma}$ and a net $(\sigma_{\gamma})_{\gamma\in\Gamma}$ in $\hat G$
such that $\sigma_{\gamma}\to \tau$ in $\hat G$,
$\Ind(x_{\lambda(\gamma)},\tau_{\lambda(\gamma)})\simeq
\Ind(x_{\lambda(\gamma)},\sigma_{\gamma})$ and
\[M_L(\Ind(z,\tau)  ,(\Ind({x_{\lambda}},\tau_{\lambda})))=
M_L(\Ind(z,\tau)  ,(\Ind({x_{\lambda(\gamma)}},\sigma_{\gamma}))).\]
\end{lemma}

\begin{proof}
By \cite[Proposition~2.3]{AS}
 there exists a subnet
$(\Ind({x_{\lambda(\delta)}},\tau_{\lambda(\delta)})_{\delta\in\Delta}$
such that
\begin{align}\label{eq-coincide}
M_L(\Ind(z,\tau)
,(\Ind(x_{\lambda},\tau_{\lambda})))&=M_U(\Ind(z,\tau)
,(\Ind(x_{\lambda(\delta)},\tau_{\lambda(\delta)})))\notag\\
&=M_L(\Ind(z,\tau)
,(\Ind(x_{\lambda(\delta)},\tau_{\lambda(\delta)}))).
\end{align}
Since the quotient map $X\times\hat G\to (X\times\hat G)/\!\sim$ is
open, there exist a subnet
$(x_{\lambda(\delta(\gamma))},\tau_{\lambda(\delta(\gamma))})_{\gamma\in\Gamma}$
of $(x_{\lambda(\delta)},\tau_{\lambda(\delta)})$ and a net
$(y_{\gamma},\sigma_{\gamma})$ in $X\times\hat G$ such that
$(x_{\lambda(\delta(\gamma))},\tau_{\lambda(\delta(\gamma))})
\sim(y_{\gamma},\sigma_{\gamma}) \to(z,\tau)$. Then
$(x_{\lambda(\delta(\gamma))},\tau_{\lambda(\delta(\gamma))})\sim
(x_{\lambda(\delta(\gamma))},\sigma_{\gamma})$ and so $
\Ind(x_{\lambda(\delta(\gamma))},\tau_{\lambda(\delta(\gamma))})=
\Ind(x_{\lambda(\delta(\gamma))},\sigma_{\gamma})$.
 By \eqref{eq-coincide}, passing to
the subnet
$(\Ind({x_{\lambda(\delta(\gamma))}},\tau_{\lambda(\delta(\gamma))}))$
does not change the relative $M_L$.
\end{proof}

\begin{lemma}\label{lem-helper} Let $(G,X)$ be a second-countable transformation
 group with $G$
abelian. Suppose that $(x_\lambda)_{\lambda\in\Lambda}$ and
$(\tau_\lambda)_{\lambda\in\Lambda}$ are nets in $X$ and $\hat G$
such that $\tau_\lambda\to 1$. Let $z\in X$ and $\xi$ a unit vector
in $L^2(G/S_z)$ and $k\in\P$. For each $\lambda$, let
$\{\xi_\lambda^{(i)}:1\leq i\leq k\}$ be an orthonormal set in
$L^2(G/S_z)$. For $1\leq i\leq k$,
\begin{equation}\label{eq-one}\langle\Ind(x_\lambda,1)(\cdot)\xi_\lambda^{(i)}\,,\, \xi_\lambda^{(i)}\rangle\to_\lambda \langle\Ind(z,1)(\cdot)\xi\,,\, \xi\rangle
\end{equation}
if and only if
\begin{equation}\label{eq-two}\langle\Ind(x_\lambda,\tau_\lambda)(\cdot)\xi_\lambda^{(i)}\,,\, \xi_\lambda^{(i)}\rangle\to_\lambda \langle\Ind(z,1)(\cdot)\xi\,,\, \xi\rangle
\end{equation}
\end{lemma}

\begin{proof}
Assume \eqref{eq-one} holds and let $b\in C_0(X)\rtimes G$. Let $\hat\alpha$ be the dual action. Since
$\Ind(x_\lambda,\tau_\lambda)(b)=\Ind(x_\lambda,1)\circ\hat\alpha_{\tau_\lambda}(b)$
and $\|\Ind(x_\lambda,1)(b-\hat\alpha_{\tau_\lambda}(b))\|\leq \|
b-\hat\alpha_{\tau_\lambda}(b) \|\to 0$, \eqref{eq-two} holds.
Similarly, if \eqref{eq-two} holds then so does \eqref{eq-one} using
$\| b-\hat\alpha_{\tau^{-1}_\lambda}(b) \|\to 0$.
\end{proof}

Recall that if $A$ is a $C^*$-algebra, $\pi\in\hat A$ and $\Omega$
is a net in $\hat A$, then $M_L(\pi,\Omega)>0$ if and only if
$\Omega$ converges to $\pi$ \cite[p.~206]{AS}.

\begin{prop}\label{prop-constant} Let $(G,X)$ be a second-countable transformation
 group with $G$ abelian. Suppose that $\Ind(x_n,\tau_n)\to \Ind(z,\tau)$. Then
\begin{align*}
M_L (\Ind(z,\tau)  , (\Ind( x_n,\tau_n) ) )
&=M_L (\Ind(z,1)  , (\Ind( x_n,\tau_n\tau^{-1}) ) )\\
&=M_L (\Ind(z,1)  , (\Ind( x_n,1 ) )\text{\ \ and}\\
M_U (\Ind(z,\tau)  , (\Ind( x_n,\tau_n) ) )
&=M_U (\Ind(z,1)  , (\Ind( x_n,\tau_n\tau^{-1}) ) )\\
&=M_U (\Ind(z,1)  , (\Ind( x_n,1 ) ) ).
\end{align*}
\end{prop}

\begin{proof}
 We have
 \[M_L\big(\Ind(z,\tau)  ,\big(\Ind( x_n,\tau_n)\big)\big)
=M_L(\Ind(z,1)\circ\hat\alpha_\tau  , (\Ind(
x_n,\tau_n\tau^{-1})\circ\hat\alpha_{\tau})).\]
Since $\hat\alpha_\tau$
is an automorphism of $C_0(X)\rtimes G$, we obtain the first
equality.

To show the second equality, fix a pure state $\phi$ associated with $\Ind(z,1)$. Let $\sigma_n=\tau_n\tau^{-1}.$ We will use the criterion of \cite[Lemma~5.2(ii)]{ASS} to prove that
\begin{equation}\label{eq-toprove}M_L(\Ind(z,1)  ,(\Ind({x_n},\sigma_n)))
=M_L(\Ind(z,1)  ,(\Ind({x_n},1)).
\end{equation}

 First, suppose that $M_L\big(\Ind(z,1)  ,\big(\Ind({x_n},1\big)\big)\geq k$ for some $k\in\P$. We will show that $M_L\big(\Ind(z,1),\big(\Ind({x_n},\sigma_n)\big)\big)\geq k$ as well. Since $\Ind(x_n,1)\to \Ind(z,1)$ we have  $k\geq 1$.   Let  $\Omega=(\Ind(x_{n_\lambda},\sigma_{n_\lambda}))_{\lambda\in\Lambda}$ be
 any subnet of $(\Ind(x_n,\sigma_n))_n$.  By passing to a further subnet and
 relabeling we may assume that $\sigma_{n_\lambda}\to 1$ in $\hat G$
 (see Lemma~\ref{lem-sub}).

Consider the subnet
$\Omega':=(\Ind(x_{n_\lambda},1))_{\lambda\in\Lambda}$ of
$(\Ind(x_n,1))_n$.
By \cite[Lemma~5.2]{ASS} there exists a further subnet
$(\Ind(x_{n_{\lambda(\mu)}},1))_{\mu\in\Upsilon}$ which has the
$k$-vector property  for $\phi$.  So by Lemma~\ref{lem-helper},
$(\Ind(x_{n_{\lambda(\mu)}},\sigma_{n_{\lambda(\mu)}}))_{\mu\in\Upsilon}$
has the $k$-vector property as well and hence $M_L\big(\Ind(z,1)
,\big(\Ind({x_n},\sigma_n)\big)\big)\geq k$. It follows that
$M_L\big(\Ind(z,1)  ,\big(\Ind({x_n},\sigma_n)\big)\big)\geq
M_L\big(\Ind(z,1)  ,\big(\Ind({x_n},1\big)\big)$.

Second, note that  $\Ind({x_n},\sigma_n)\to\Ind(z,1)$. Suppose that $M_L\big(\Ind(z,1)  ,\big(\Ind({x_n},\sigma_n)\big)\big)\geq k$ for some $k\geq 1$. Let $\Omega=(\Ind(x_{n_\lambda},1))_{\lambda\in\Lambda}$ be any subnet of $(\Ind(x_n,1))_n$. Consider $\Omega':=(\Ind(x_{n_\lambda},\sigma_{n_\lambda}))_{\lambda\in\Lambda}$; by passing to a further subnet and relabeling we may assume that $\sigma_{n_\lambda}\to 1$.  By \cite[Lemma~5.2]{ASS}, there exists a subnet $(\Ind(x_{n_{\lambda(\mu)}},\sigma_{n_{\lambda(\mu)}}))_{\mu\in\Upsilon}$  of $\Omega'$ with the $k$-vector property for $\phi$. So by Lemma~\ref{lem-helper}, $(\Ind(x_{n_{\lambda(\mu)}},1))_{\mu\in\Upsilon}$ has the $k$-vector property as well and hence $M_L\big(\Ind(z,1)  ,\big(\Ind({x_n},1)\big)\big)\geq k$. Thus \eqref{eq-toprove} holds.

The proof of the corresponding statement about the upper multiplicities is similar,
using the criterion of  \cite[Lemma~5.2(i)]{ASS} .
\end{proof}

The next theorem is an analogue of \cite[Theorem~2.3]{AD} and
\cite[Theorem~2.1]{AaH2}.  In Corollary~\ref{cor-lowermult}, we
will show that the three hypotheses of Theorem~\ref{thm-lowermult} can be
satisfied under a suitable assumption of $k$-times convergence.

\begin{thm}\label{thm-lowermult}
Let $(G,X)$ be a second-countable transformation group with $G$
abelian. Let $k\in\P$, $z\in X$. Let $(x_n)_n$ be a sequence in $X$
such that $S_{x_n}\to S_z$.  Suppose that there exists a compact,
symmetric neighbourhood $W$ of the identity in $G$ and $k$ sequences
$(t_n^{(1)})_n,(t_n^{(2)})_n,\dots,(t_n^{(k)})_n$ in $G$ such that
\begin{enumerate}
\item\label{thm-lowermult-1} $t_n^{(j)}\cdot x_n\to z$ for $1\leq j\leq k$;
\item\label{thm-lowermult-2} there exists $N$ such that $n\geq N$ implies
$Wt_n^{(j)}S_{x_n}\cap Wt_n^{(i)}S_{x_n}=\emptyset$ for $1\leq i<j\leq k$;
\item\label{thm-lowermult-3} $\mu\big(WS_z\setminus(\interior W)S_z\big)=0$.
\end{enumerate}
If  $\Ind(x_n,\tau_n)\to \Ind(z,\tau)$  as $n\to\infty$ then $M_L(\Ind(z,\tau)  ,(\Ind(x_n,\tau_n)))\geq k$.
\end{thm}

\begin{proof} We will prove that $M_L(\Ind(z,1),(\Ind(x_n,1)))\geq k$; since $G$ is abelian this suffices by Proposition~\ref{prop-constant}.

For $x\in X$, let $q_x:G\to G/S_x$ be the quotient map and for $s\in G$ set
\[
\eta(q_z(s)):=\nu_z\big(q_z(WS_z)\big)^{-1/2}\Chi_{q_z(WS_z)}(q_z(s));\]
for $n\geq 1$ and $1\leq j\leq k$ set
\[
\eta_n^{(j)}(q_{x_n}(s)):=\nu_{x_n}\big(q_{x_n}(WS_{x_n})\big)^{-1/2}\Chi_{q_{x_n}(WS_{x_n})}
(q_{x_n}(st_n^{(j)-1})).
\]
Then $\eta$ is a unit vector in $L^2(G/S_z,\nu_z)$, and, for each
$n\geq 1$ and $1\leq j\leq k$,  $\eta_n^{(j)}$ is a  unit vector in
$ L^2(G/S_{x_n},\nu_{x_n})$. Note that $st_n^{(j)-1}\in WS_{x_n}$
and $st_n^{(i)-1}\in WS_{x_n}$ if and only if $s\in
Wt_n^{(j)}S_{x_n}\cap Wt_n^{(i)}S_{x_n}$. But if $n\geq N$ and
$i\neq j$ then  $Wt_n^{(j)}S_{x_n}\cap Wt_n^{(i)}S_{x_n}=\emptyset$ by hypothesis \eqref{thm-lowermult-2},
and hence $\langle \eta_n^{(j)}\,,\, \eta_n^{(i)}\rangle =0$.

Fix $f\in C_c(G\times X)\subset C_0(X)\rtimes G$  of the form
$f(s,x)=h(s)g(x)$ where $h\in C_c(G)$ and $g\in C_c(X)$. We will compute
 \[\Psi_n^{(j)}(f):=\langle \big( \Ind(x_n,1)f\big)\eta_n^{(j)}\,,\, \eta_n^{(j)}\rangle\]
for $1\leq j\leq k$.
 To simplify the formulas, we write  $q_n$ for
 $q_{x_n}$, $\Chi_n$ for $\Chi_{q_n(WS_{x_n})}$ and $C_n$ for
 $\nu_{x_n}(q_n(WS_{x_n}))^{-1}$.
 We compute using the formulas from \cite[p.~1216]{aH}:
\begin{align*}
\Psi_n^{(j)}(f)&=\int_{G/S_{x_n}}\big((\Ind(x_n,1)f)\eta_n^{(j)}
\big)(\dot v)\overline{ \eta_n^{(j)}(\dot v)}\, d\nu_{x_n}(\dot v)\notag\\
&=C_n\int_{G/S_{x_n}}\int_G h(vu^{-1})g(v\cdot x_n)
\Chi_n(q_n(ut_n^{(j)-1}))\, d\mu(u)\Chi_n (q_n(vt_n^{(j)-1}))\,
d\nu_{x_n}(\dot v)\notag\
\intertext{which, via changes of variables
$s=ut_n^{(j)-1}$ and $\dot r=q_n(vt_n^{(j)-1})$, is}
&=C_n \int_{G/S_{x_n}}\int_G h(rs^{-1})g(rt_n^{(j)}\cdot x_n)\Chi_n(\dot s)\Chi_n(\dot r)\,d\mu(s)\, d\nu_{x_n}(\dot r)\notag\\
&=C_n \int_{G/S_{x_n}}\int_G h(u)g(rt_n^{(j)}\cdot x_n)\Chi_n(q_n(u^{-1}r))\Chi_n(\dot r)\,d\mu(u)\, \,d \nu_{x_n}(\dot r)\notag\\
&=C_n\int_{\dot r\in q_n(WS_{x_n})}g(rt_n^{(j)}\cdot x_n)\left(\int_{u\in rWS_{x_n}} h(u)\, d\mu(u)\right)\,d \nu_{x_n}(\dot r).\notag
\end{align*}
For each $n\geq 1$,
\[
F_n^j(r)=g(rt_n^{(j)}\cdot x_n)\int_{u\in rWS_{x_n}} h(u)\, d\mu(u)
\]
is the formula for a function $F_n^j$ which is constant on cosets of
$S_{x_n}$, and the above calculation shows
\begin{equation}
\Psi_n^{(j)}(f)=C_n\int_{\dot r\in q_n(WS_{x_n})}F_n^j(r)\,d \nu_{x_n}(\dot r).\label{eq-long}
\end{equation}
Note, for later use, that the $F_n^j$ are uniformly
bounded by $\Vert g\Vert_{\infty}\Vert h\Vert_1$.

Since $S_{x_n}\to S_z$, it follows from Lemma~\ref{lem-pointwise} and hypothesis \eqref{thm-lowermult-3} that $\Chi_{WS_{x_n}}\to \Chi_{WS_z}$ almost everywhere. By the invariance of $\mu$, for each
$r\in G$ we have $\Chi_{rWS_{x_n}}\to \Chi_{rWS_z}$ almost everywhere.
 By hypothesis \eqref{thm-lowermult-1} and the continuity of $g$, for all $1\leq j\leq k$
and $r\in G$,  $g(rt_n^{(j)}\cdot x_n)\to g(r\cdot z)$ as
$n\to\infty$. Since $h\in L^1(G,\mu)$ it follows that, for all
$r\in G$,
\begin{equation}\label{eq-F}
F_n^j(r)\to g(r\cdot z)\int_{u\in rWS_{z}}h(u)\, d\mu(u)
\end{equation}
as $n\to\infty$.

 By \cite[Proposition~2.18(ii)]{W} or \cite[Proposition~H.17(a)]{tfb^2} we can choose a ``cut-down generalised Bruhat approximate cross-section'' $b\in C_c(G\times \Sigma)$ such that
\[
\int_{H}b(rt,H)\, d\alpha_H(t)=1
\]
for $r\in WH$.  Thus
\begin{align*}
\eqref{eq-long}
&=C_n\int_{G/S_{x_n}}\chi_{WS_{x_n}}(r)F_n^j(r)\,d \nu_{x_n}(\dot r)\\
&=C_n\int_{G/S_{x_n}}\chi_{WS_{x_n}}(r)F_n^j(r)\int_{S_{x_n}}b(rt,S_{x_n})\,
d\alpha_{x_n}(t)\,d \nu_{x_n}(\dot r)\\
&=C_n\int_{G/S_{x_n}}\int_{S_{x_n}}\chi_{WS_{x_n}}(rt)F_n^j(rt)b(rt,S_{x_n})\,
d\alpha_{x_n}(t)\,d \nu_{x_n}(\dot r)\\
&=C_n\int_G\Chi_{WS_{x_n}}(r)F_n^j(r)b(r,S_{x_n})\, d\mu(r)
\end{align*}
using \eqref{quotientmeasure2}.
Similarly,
\begin{align*}
C_n^{-1}&=\nu_{x_n}(q_n(WS_{x_n}))=\int_{G/S_{x_n}}\Chi_{WS_{x_n}}(r) \int_{S_{x_n}}b(rt,S_{x_n})\, d\alpha_{x_n}(t)\, d\nu_{x_n}(\dot r)\\
&=\int_{G/S_{x_n}}\int_{S_{x_n}} \Chi_{WS_{x_n}}(rt)b(rt,S_{x_n})\,
d\alpha_{x_n}(t)\, d\nu_{x_n}(\dot r)=\int_G
\Chi_{WS_{x_n}}(r)b(r,S_{x_n})\,
 d\mu(r).
\end{align*}
So we have shown that $\Psi_n^{(j)}(f)$ is the product of
\begin{equation}
C_n=\left(\int_G  \Chi_{WS_{x_n}}(r)b(r,S_{x_n})\,
d\mu(r)\right)^{-1}\label{eq-1}
\end{equation}
and
\begin{equation}
\int_G\Chi_{WS_{x_n}}(r)F_n^j(r)b(r,S_{x_n})\,
d\mu(r)\label{eq-3}.
\end{equation}
An almost identical calculation shows that
\[\Psi(f):=\langle \big( \Ind(z,1)f\big)\eta\,,\, \eta\rangle\] is the product of
\begin{equation}
C:=\left(\int_G  \Chi_{WS_z}(r)b(r,S_z)\,
d\mu(r)\right)^{-1}\label{eq-4}
\end{equation}
and
\begin{equation}
\int_G \left(\Chi_{WS_z}(r)g(r\cdot z)b(r,S_z)\int_{u\in
rWS_z}h(u)\, d\mu(u)\right)\, d\mu(r) \label{eq-6}.
\end{equation}

We have noted above that  $\Chi_{WS_{x_n}}\to \Chi_{WS_z}$ almost everywhere. Since $b$ is continuous on $G\times\Sigma$, $b(\cdot, S_{x_n})\to b(\cdot, S_z)$. So the integrand in \eqref{eq-1}
converges pointwise almost everywhere to the integrand in \eqref{eq-4}, and hence
$C_n$ converges to $C$ by the Dominated Convergence Theorem. (For an
$L^1$-dominant, let $L$ be the compact subset of $G$ obtained by
projecting the support of $b$ and then take $\Vert
b\Vert_{\infty}\Chi_L$.) Using \eqref{eq-F},
the integrand in \eqref{eq-3} converges pointwise almost everywhere to the
integrand in \eqref{eq-6}, and hence it follows from the Dominated
Convergence Theorem that $\Psi_n^{(j)}(f)\to \Psi(f)$ for $1\leq
j\leq k$. (For an $L^1$-dominant, we may take $\Vert
g\Vert_{\infty}\Vert h\Vert_1\Vert
b\Vert_{\infty}\Chi_L$.) Since the linear span of
such $f$ is norm-dense in $C_0(X)\rtimes G$ and the $\Psi_n^{(j)}$
and $\Psi$ have norm one, it follows that $\Psi_n^{(j)}\to \Psi$ in
the weak$^*$-topology for $1\leq j\leq k$.

 Suppose that $({\rm
Ind}(\epsilon_{x_{n_\lambda}},1))_{\lambda\in\Lambda}$ is a subnet
of $({\rm Ind}(x_n,1))$.  Then there exists $\lambda_0\in
\Lambda$ such that $n_{\lambda}\geq N$ whenever $\lambda\geq
\lambda_0$.   So the calculations above give, for each $\lambda\geq
\lambda_0$, $k$ mutually orthogonal pure states
$\Psi_{n_\lambda}^1,\dots,\Psi_{n_\lambda}^k$ associated with ${\rm
Ind}(\epsilon_{x_{n_\lambda}},1)$ such that
$\lim_\lambda\Psi_{n_\lambda}^j=\Psi$ for $1\leq j\leq k$. It now
follows from \cite[Lemma 5.2(ii)]{ASS} that
$M_L\big(\Ind(z,1)
,\big(\Ind(x_n,1)\big)\big)\geq k$.
\end{proof}

Before moving to Corollary~\ref{cor-lowermult}, we need to consider
the measure-theoretic hypothesis (3) in Theorem~\ref{thm-lowermult}.

\begin{lemma}\label{lemma_measureW} Let $G$ be a locally compact,
abelian group with Haar measure $\mu$ and let $S$ be a closed
subgroup of $G$ such that $S$ is either countable or compact. Then
there exists a compact symmetric neighbourhood $W$ of the
identity in $G$ such that $\mu\big(WS\setminus(\interior
W)S\big)=0$.
\end{lemma}

\begin{proof}
We may assume that $G=H\times\R^k$, for some $k\geq 0$, where $H$ is
a locally compact abelian group containing a compact open subgroup
$K$ \cite[Theorem~4.2.1]{DE}. If $k=0$, we may simply take $W=K$.
Assuming that $k\geq1$, let $W=K\times B$ where $B$ is the closed
unit ball of $\R^k$. Then $\interior W= K\times \interior B$ and
$W\setminus \interior W=K\times(B\setminus \interior B)$ which has
$\mu$-measure zero since the Lebesgue measure of $B\setminus
\interior B$ is zero.

Suppose that $S$ is countable. Then $WS\setminus(\interior
W)S\subseteq (W\setminus(\interior W))S$ and the right-hand side has
$\mu$-measure zero by the invariance of $\mu$ and the countability
of $S$.

Suppose, instead, that $S$ is compact. Then the image of $S$ under
the projection $p:G=H\times \R^k\to\R^k$ is a compact subgroup of
$\R^k$ and hence is $\{0\}$. Thus $S=T\times\{0\}$ where $T$ is a
subgroup of $H$. Then $WS\setminus (\interior
W)S=KT\times(B\setminus\interior B)$ which
   has $\mu$-measure zero because $B\setminus\interior B$ has zero Lebesgue measure.
\end{proof}

A more complicated argument shows that if $S$ is a closed subgroup
of a second-countable, locally compact, abelian group $G=H\times\R^k$
such that the image  $p(S)$ of the projection of $S$ into $\R^k$ is closed,  then there exists a compact
symmetric neighbourhood $W$ of the identity in $G$ such that
$\mu\big(WS\setminus(\interior W)S\big)=0$. We omit the details as
we will not use that result here.

In several places we assume that $S_z$ is compact (for example in Proposition~\ref{prop-tsoc2}
and \S\S\ref{measure} and \ref{main}), and so we restrict to this case in the following result.

\begin{cor}\label{cor-lowermult}
Let $(G,X)$ be a second-countable transformation group with $G$
abelian. Let $z\in X$ with $S_z$ compact. Let $k\in\P$ and let $(x_n)_n$ be a sequence in $X$ converging $k$-times in $X/G$ to $z$ such that $S_{x_n}\to S_z$.
If  $\Ind(x_n,\tau_n)\to \Ind(z,\tau)$  as $n\to\infty$ then $M_L\big(\Ind(z,\tau)  ,\big(\Ind(x_n,\tau_n)\big)\big)\geq k$.
\end{cor}

\begin{proof}
Since $S_z$ is compact, by Lemma~\ref{lemma_measureW}
that there exists a compact symmetric neighbourhood $W$ of the
identity in $G$ such that $\mu\big(WS_z\setminus(\interior
W)S_z\big)=0$.

By the $k$-times convergence there exist $k$ sequences $(t_n^{(1)})_n,(t_n^{(2)})_n,\dots,(t_n^{(k)})_n$ in $G$ such that
$t_n^{(i)}\cdot x_n\to z$ as $n\to\infty$ for $1\leq i\leq k$ and
$t_n^{(j)}(t_n^{(i)})^{-1}S_{x_n}\to\infty$ as
$n\to\infty$ for $i\neq j$.
Since $W^2$ is compact, there exists $N\in \P$ such that $n\geq N$ implies $t_n^{(j)}t_n^{(i)-1}S_{x_n}\cap W^2=\emptyset$.   Thus $Wt_n^{(j)}S_{x_n}\cap Wt_n^{(i)}S_{x_n}=\emptyset$ for $1\leq i<j\leq k$ when $n\geq N$. The result now follows from
Theorem~\ref{thm-lowermult}.
\end{proof}

Next we use well-known transformation groups to give  examples of  transformation groups with a sequence $x_n\to z$ where the stability subgroup at $z$ is not compact (so that Corollary~\ref{cor-lowermult} does not apply), but where we can still verify that all the hypotheses of Theorem~\ref{thm-lowermult} hold.

\begin{examples} Let $(\R, Y)$ be Green's free non-proper transformation group \cite[pp.~95-96]{green1}: the space $Y$
is a closed subset of $\R^3$ and  consists of
countably many orbits, with orbit representatives
$y_0=\mathbf{0}:=(0,0,0)$ and $y_n=(2^{-2n},0,0)$ for $n=1,2,\dots$.
The action of $\R$ is given by
$s\cdot y_0=(0,s,0)$ for all $s$; and for $n\geq 1$,
\[
s\cdot y_n =
\left\{\begin{array}{ll} (2^{-2n},s,0)& \text{if}\,\,s\leq n;
\\(2^{-2n}-(\frac{s-n}{\pi})2^{-2n-1}, n\cos(s-n), n\sin(s-n) )&\text{if}\,\,n<s<n+\pi;
\\(2^{-2n-1},s-\pi-2n,0) &\text{if}\,\, s\geq n+\pi. \end{array}
\right.
\]
Green's action is free and each orbit consists of two vertical lines joined by an arc of a helix situated on a cylinder of radius $n$;
the action moves $y_n$ along the vertical lines at unit speed,
and along the arc at radial speed.

Next let $(\R,\C)$ be the non-free transformation group of
\cite[Example~5.4]{W}. Here $\R$ acts on  $\C$ by fixing
 the origin, and if $w\neq 0$
then $r\cdot w= e^{2\pi i (\frac{r}{|w|})}w$.
The orbits are  concentric circles about the origin and
the stability subgroups are
\[
S_w =
\left\{\begin{array}{ll}  \R& \text{if}\,\,w=0;
\\ |w|\Z &\text{if}\,\, w\neq 0, \end{array} \right.
\]
and vary continuously.

Now we are ready for our examples illustrating  the hypotheses of Theorem~\ref{thm-lowermult}.
\begin{enumerate}
\item  Let $G=\R\times\R$ act on $Y\times\C$ by $(s,r)\cdot(y,w)=(s\cdot y,r\cdot w)$.  Then the stability subgroups are
\[
S_{y,w} =
\left\{\begin{array}{ll}  \{0\}\times\R& \text{if}\,\,w=0;
\\ \{0\}\times|w|\Z &\text{if}\,\, w\neq 0, \end{array} \right.
\]
and vary continuously.  Let $w_n\to 0$ in $\C$ and consider the sequence $x_n=(y_n, w_n)\to (\mathbf{0},0)$ in $Y\times\C$, where $y_n$ are the orbit representatives for the action of $\R$ on $Y$ described above.  The stability subgroup at $(\mathbf{0},0)$ is the non-compact group $\{0\}\times\R$. We claim that the hypotheses of Theorem~\ref{thm-lowermult} hold with
\[
t_n^{(1)}=(0,0),\ t_n^{(2)}=(2n+\pi,0),\text{\ and\ } W=[-1,1]\times[-1,1].
\]
First,  $t_n^{(1)}\cdot (y_n, w_n)=(y_n, w_n)\to (\mathbf{0},0)$ and $t_n^{(2)}\cdot (y_n, w_n)=\big((2^{-2n-1},0,0), 0\big)\to (\mathbf{0},0)$.  Second,
\begin{gather*}W+t_n^{(1)}+S_{(y_n, w_n)}=[-1,1]\times\{t+|w_n|\Z: t\in[-1,1]\}\text{ and}\\
W+t_n^{(2)}+S_{(y_n, w_n)}=[-1+2n+\pi,1+2n+\pi]\times\{t+|w_n|\Z: t\in[-1,1]\},
\end{gather*} and hence $(W+t_n^{(1)}+S_{(y_n, w_n)})\cap (W+t_n^{(2)}+S_{(y_n, w_n)})=\emptyset$ when $n\geq 1$.  Third, $WS_{(\mathbf{0},0)}=[-1,1]\times\R$ and $(\interior W)S_{(\mathbf{0},0)}=(-1,1)\times\R$, and hence $WS_{(\mathbf{0},0)}\setminus ((\interior W)S_{(\mathbf{0},0)})$ has measure zero in $\R\times\R$.  So the three hypotheses of Theorem~\ref{thm-lowermult} hold as claimed.
\item Let $G=\R\times\R$ act on $Y$ by $(s,r)\cdot y=s\cdot y$.  Then the stability subgroup at $y$ is $\{0\}\times \R$, and since the stability subgroups are constant they vary continuously.  Consider $y_n\to \mathbf{0}$ in $Y$.  Then, for the same  reasons as in (1), the hypotheses of Theorem~\ref{thm-lowermult} hold with
\[
t_n^{(1)}=(0,0),\ t_n^{(2)}=(2n+\pi,0),\text{\ and\ } W=[-1,1]\times[-1,1].
\]
\end{enumerate}
\end{examples}

\section{Measure accumulation and upper bounds on multiplicities}\label{measure}

Throughout this section $G$ is assumed to be abelian. Here we use bounds on measure accumulation to find upper bounds on $M_L(\pi,(\pi_n))$ where $\pi_n$ and $\pi$ are induced representations of $C_0(X)\rtimes G$.  Theorem~\ref{thm-M} has the same hypothesis as Proposition~\ref{thm-Msquared} but a stronger conclusion; in particular its  proof uses that $M_L(\Ind(z,\tau),(\Ind(x_n,\tau_n)))$ is finite by  Proposition~\ref{thm-Msquared}.

\begin{prop}\label{thm-Msquared}
Suppose that $(G,X)$ is a second-countable transformation group with
$G$ abelian. Let $z\in X$  and let $(x_n)_{n\geq 1}$ be
a sequence in $X$ such that $S_{n_n}\to S_z$. Assume that $G\cdot z$ is locally closed in $X$
and that $S_z$ is compact.  Let $M\in\R$ with $M\geq 1$, and suppose that there exists an open
neighbourhood $V$ of $z$ in $X$ such that $\phi_z^{-1}(V)$ is
relatively compact and
\begin{equation}\label{eq-data}
\nu_{x_n}(q_{x_n}(\phi_{x_n}^{-1}(V)))\leq M\nu_z(q_z((\phi_z^{-1}(V))))
\end{equation}
frequently. Then $M_L(\Ind(z,\tau),(\Ind(x_n,\tau_n)))\leq\lfloor M^2\rfloor$ for any $\tau,\tau_n\in \hat G$.
\end{prop}

\begin{proof}
We  may assume that $\Ind(x_n,\tau_n)\to \Ind(z,\tau)$, since $M_L(\Ind(z,\tau),(\Ind(x_n,\tau_n)))=0$ otherwise. So by Proposition~\ref{prop-constant} it suffices to show that $M_L(\Ind(z,1),(\Ind(x_n,1)))\leq\lfloor M^2\rfloor $.
 Since $M_L(\Ind(z,1),(\Ind(x_n,1)))\leq M_L(\Ind(z,1),(\Ind(x_{n_i},1))$ for any subsequence $(\Ind(x_{n_i},1))$, we may assume that \eqref{eq-data} holds for all $n$.

Next, we  claim that we may, by passing to a further subsequence, assume that $x_n\to z$.   Note that $\ker\Ind(x_n,1)\to \ker\Ind(z,1)$ in $\Prim(C_0(X)\rtimes G)$. Since  $\Prim(C_0(X)\rtimes G)$ is homeomorphic to $(X\times \hat G)/\!\sim$ and since the quotient map $X\times\hat G\to  (X\times \hat G)/\!\sim$ is open \cite[Theorem~5.3]{W2}, there exists a subsequence $(x_{n_i},1)$ of $(x_n,1)$ and  $(y_i,\sigma_i)\in X\times\hat G$ such that $(x_{n_i},1)\sim(y_i,\sigma_i)\to (z,1)$ in $X\times \hat G$.  Thus $\overline{G\cdot y_i}=\overline{G\cdot x_{n_i}}$ and $\sigma_i|_{S_{x_{n_i}}}=1$.
Let $(N_k)$ be a decreasing basic sequence of open neighbourhoods of $z$ in $X$.  There exists a subsequence $(y_{i_k})$ such that $y_{i_k}\in N_k$.  Since $\overline{G\cdot y_{i_k}}=\overline{G\cdot x_{n_{i_k}}}$ there exists $g_k\in G$ such that $g_k\cdot x_{n_{i_k}}\in N_k$. Hence $g_k\cdot x_{n_{i_k}}\to z$.
By \cite[Corollary~4.8]{W2}
$\Ind(x_{n_{i_k}},1)$ and $\Ind(g_k\cdot x_{n_{i_k}},1)$ are unitarily equivalent, and by the invariance of the measure we can replace $x_{n_{i_k}}$ with $g_k\cdot x_{n_{i_k}}$ in \eqref{eq-data}. So we may assume that $x_n\to z$ as claimed.

Now we will adapt the proof of \cite[Theorem~3.1]{AaH}.
Fix $\epsilon>0$ such that $M^2(1+\epsilon)^4<\lfloor M^2\rfloor +1$.
We will build a function $D\in C_c(G\times X)$
such that $\Ind(z,1)(D^**D)$ is a rank-one projection and
\[
\tr(\Ind(x_n,1)(D^**D))<M^2(1+\epsilon)^4<\lfloor M^2\rfloor +1
\] eventually.
(The function $D$ is similar to the ones used in  \cite[Proposition~4.5]{aH}, \cite[Proposition~4.2]{W2} and \cite[Theorem~3.1]{AaH}.)
By the generalised lower semi-continuity result of \cite[Theorem~4.3]{AS} we will have
\begin{align*}
\liminf_n \tr(\Ind(x_n,1)(D^**D))&\geq
M_L(\Ind(z,1),(\Ind(x_n,1)))\tr(\Ind(z,1)(D^**D))\\&=
M_L(\Ind(z,1),(\Ind(x_n,1))),
\end{align*}
and the theorem will follow.

Let $\delta>0$ such that
\begin{equation*}
\delta<\frac{\epsilon\nu_z(q_z(\phi_z^{-1}(V)))}{1+\epsilon}<\nu_z(q_z(\phi_z^{-1}(V))).
\end{equation*}
By the regularity of the measure $\nu_z$ there exists a compact subset $W$ of $G/S_z$ such that  $W\subset q_z(\phi_z^{-1}(V))$ and
\begin{equation*}
0<\nu_z(q_z(\phi_z^{-1}(V)))-\delta<\nu_z(W).
\end{equation*}
Since $W$ is compact, there is a compact neighbourhood $W_1$ of
$W$ contained in the open set $q_z(\phi_z^{-1}(V))$ and a continuous function
$g:G/S_z\to[0,1]$ such that $g$ is identically one on $W$ and is
identically zero off the interior of $W_1$. Then
\begin{equation*}
\nu_z(q_z(\phi_z^{-1}(V)))-\delta<\nu_z(W)\leq \int_{G/S_z} g(\dot u)^2\, d\nu_z(\dot u)=\|g\|_{2,z}^2,
\end{equation*}
and hence
\begin{equation}\label{eq-estimate}
\frac{\nu_z(q_z(\phi_z^{-1}(V)))}{\|g\|_{2,z}^2}<
1+\frac{\delta}{\|g\|_{2,z}^2}<1+\frac{\delta}{\nu_z(q_z(\phi_z^{-1}(V)))-\delta}<1+\epsilon.
\end{equation}
Since $G\cdot z$ is locally closed in $X$ it follows from
\cite[Theorem~1]{Gli}, applied to the locally compact Hausdorff
transformation group $(G, G\cdot z)$, that $\phi_z:G\to G\cdot z$ induces a
homeomorphism $\dot s\mapsto s\cdot z$ of $G/S_z$ onto $G\cdot z$.
 So there is a continuous
function $g_1:W_1\cdot z\to[0,1]$ such that $g_1(\dot u\cdot z)=g(\dot u)$ for
$\dot u\in W_1$. Since $W_1\cdot z$ is a compact subset of the locally
compact Hausdorff space $X$, it follows from Tietze's Extension
Theorem (applied to the one-point compactification of $X$ if
necessary) that $g_1$  can be extended to a continuous function
$g_2:X\to [0,1]$. Because $W_1\cdot z$ is a compact subset of the
open set $V$, there exists a compact neighbourhood $P$ of $W_1\cdot
z$ contained in $V$ and a continuous function $h:X\to [0,1]$ such
that $h$ is identically one on $W_1\cdot z$ and is identically zero
off the interior of $P$.
 Note that $h$ has compact support contained in $P$. We set
\[
f(x)=h(x)g_2(x).
\]
Then $f\in C_c(X)$ with $0\leq f\leq 1$ and  $\supp f\subset\supp
h\subset P\subset V$.  Set $\tilde f_z(\dot s)=f(s\cdot z)$ so that $\tilde f:G/S_z\to X$.   Note that
\begin{equation}\label{eq-reciprocal}
 \|\tilde f_z\|^2_{2,z}=\int_{G/S_z} \tilde f(\dot u\cdot z)^2\, d\nu_z(\dot u)=\int_{G/S_z} h(\dot u\cdot z)^2g_2(\dot u\cdot z)^2\, d\nu_z(\dot u)\geq\int_{W_1} g(\dot u)^2\, d\nu_z(\dot u) =\|g\|^2_{2,z}
\end{equation}
since $h$ is identically one on $W_1\cdot z$ and the support of  $g$
is contained in $W_1$. We now set
\[
F(x)=\frac{f(x)}{\|\tilde f_z\|_{2,z}}.
\]
Now $F\in C_c(X)$ and  $F_x(s)=F(s\cdot x)\neq 0$ implies
$s\in\phi_x^{-1}(V)$ by our choice of $h$. Since $\phi_z^{-1}(V)$ is
relatively compact, $\supp F_z$ is compact. Write $\tilde F_z(\dot
s)=F(s\cdot z)$ and note that $\|\tilde F_z\|_{2,z}=1$.

Recall that $S_z$ is compact by assumption. Choose $b\in C_c(G\times X)$ such that $0\leq b\leq 1/\alpha_z(S_z)$ and $b$ is identically
$1/\alpha_z(S_z)$ on the set $(\supp F_z)S_z(\supp F_z)^{-1}\times\supp F$; we may obtain that $\supp b\subset N\times X$ where $N=N^{-1}$ is a compact subset of $G$ containing $S_z$.
Set
\[
B(r,x)=F(x)F(r^{-1}\cdot x)b(r^{-1},x)\quad \text{and}\quad
D=\frac{1}{2}(B+B^*).
\]
We have
 \begin{align*}
(\Ind(x,1)(B)\xi)(\dot s)
&=\int_G B(r,s\cdot x)\xi(r^{-1}\dot s)
  \, d\mu(r)\\
&=\int_G F(s\cdot x)F(r^{-1}s\cdot x)b(r^{-1},s\cdot x)
  \xi(r^{-1}\dot s)
  \, d\mu(r)\\
&=F(s\cdot x)
  \int_G F(u\cdot x)b(us^{-1},s\cdot x) \xi(\dot u)
  \,d\mu(u)\\
&=F(s\cdot x)
  \int_{G/S_x} F(u\cdot x)\int_{S_x}
 b(uts^{-1},s\cdot x)
  \, d\alpha_x(t)\xi(\dot u)\, d\nu_x(\dot u)
\end{align*}
so that
\begin{align*}
&(\Ind(x,1)(D)\xi)(\dot s)\\
&
=\frac{1}{2}F(s\cdot x)
 \int_{G/S_x} F(u\cdot x)\left(\int_{S_x}
\big( b(us^{-1}t,s\cdot x) +
 b(su^{-1}t,u\cdot x)\big) \ d\alpha_x(t)\right)\xi(\dot u)
\, d\nu_x(\dot u)
\end{align*}
since $G$ is abelian.

If $F(s\cdot z)$ and $F(u\cdot z)$ are nonzero then $s,u\in \supp F_z$, and hence $b(uts^{-1},s\cdot z)+b(stu^{-1}, u\cdot z)=2/\alpha_z(S_z)$ for all $t\in S_z$.  It follows
that
\[(\Ind(z,1)(D)\xi)(\dot s)=F(s\cdot z)
 \int_{G/S_z} F(u\cdot z)\xi(\dot u)
\, d\nu_z(\dot u)=(\xi, \tilde F_z)\tilde F_z(\dot s).\] Thus $\Ind(z,1)(D)$, and
hence $\Ind(z,1)(D^*D)$, is the rank-one projection determined by the unit vector
$\tilde F_z\in L^2(G/S_z,\nu_z)$.

Recall that we are assuming that \eqref{eq-data} holds for all $n$ and set
$
E_n=\{s\in G: F(s\cdot x_{n})\neq 0\}
$.
Then each $q_{x_{n}}(E_n)$ is open, hence measurable, with
\begin{equation}\label{eq-measuremiracle}\nu_{x_n}(q_{x_n}(E_n))\leq\nu_{x_n}(q_{x_n}(\phi_{x_{n}}^{-1}(V)))\leq M\nu_z(q_z(\phi_z^{-1}(V)))<\infty.\end{equation}

Note that $\Ind(x_{n},1)(D)$ is a kernel operator with kernel
\[K_n(\dot s, \dot u):=\frac{1}{2}F(s\cdot x_{n}) F(u\cdot x_n)\left(\int_{S_{x_{n}}}(b(us^{-1}t,s\cdot x_{n})+b(su^{-1}t,
u\cdot x_{n})\, d\alpha_{x_{n}}(t)\right).
\]
To see that $\Ind(x_{n},1)(D)$ is a Hilbert-Schmidt operator on $L^2(G/S_{x_{n}})$, we need to see that $K_n$ is in $L^2(G/S_{x_{n}}\times G/S_{x_{n}})$. The support of $K_n$ is contained in $q_{x_{n}}(E_n)\times q_{x_{n}}(E_n)$, which has finite measure by \eqref{eq-measuremiracle}. Note that $K_n$ is continuous, hence measurable, since $F$ is continuous and $b\in C_c(G\times X)$.  To see that $K_n$ is bounded,
set
\[
\Upsilon_n(s,u):=\int_{S_{x_{n}}}(b(us^{-1}t,s\cdot x_{n})+b(su^{-1}t,
u\cdot x_{n}))\, d\alpha_{x_{n}}(t),
\]
and note that $\Upsilon_n$ is constant on $S_{x_{n}}\times S_{x_{n}}$-cosets. Recall that $0\leq b\leq1/\alpha_z(S_z)$ and that $\supp b\subset N\times X$ where $N=N^{-1}$ and $S_z\subset N$.  If $us^{-1}t\notin N$ for all $t\in S_{x_{n}}$ then $\Upsilon(s,u)=0$. If $us^{-1}t_0\in N$ for some $t_0\in S_{x_{n}}$ then we may assume that $us^{-1}\in N$ (because $\Upsilon_n(s,u)=\Upsilon_n(t_0^{-1}s,u)=\Upsilon_n(t_0^{-1}s,ut_0^{-1})$). Thus
\[
\Upsilon_n(s,u)\leq \frac{2}{\alpha_z(S_z)}\alpha_{x_{n}}(\{t\in S_{x_{n}}:us^{-1}t\in N\})\leq  \frac{2}{\alpha_z(S_z)}\alpha_{x_{n}}(S_{x_{n}}\cap N^2).
\]
Let $\eta\in C_c(G)^+$ such that $\eta$ is identically one on $N^2$.  It follows from our choice of continuous Haar measures on the closed subgroups of $G$ that $H\mapsto \int_H \eta(t)\, d\alpha_H(t)$ is a continuous function on $\Sigma$. Since   $S_{x_n}\to S_z$ by assumption, there exists $n_0$ such that, for $n\geq n_0$,
\begin{equation}\label{eq-upsilon}
\Upsilon_n(s,u)
\leq   \frac{2}{\alpha_z(S_z)}\int_{S_{x_{n}}}\eta(t)\,d\alpha_{x_{n}}(t)\leq   \frac{2}{\alpha_z(S_z)}\int_{S_z}\eta(t)\,d\alpha_z(t)(1+\epsilon)=2(1+\epsilon).
\end{equation}
Hence  $0\leq K_n(\dot s,\dot u)\leq \|F\|^2_\infty(1+\epsilon)$ when $n\geq n_0$.

Let $n\geq n_0$.  Then  $\Ind(x_{n},1)(D)$ is the self-adjoint Hilbert-Schmidt operator  with kernel $K_n$.
It follows that $\Ind(x_{n},1)(D^**D)$ is a trace-class operator with
\[\tr(\Ind(x_{n},1)(D^**D))=\|K_n\|_{2,x_{n}}^2\]
(see, for example, \cite[Proposition~3.4.16]{ped}). To estimate the
trace we note, using \eqref{eq-measuremiracle} and
\eqref{eq-reciprocal}, that
\begin{equation}\label{eq-moved}
\int_{G/S_{x_n}} F(s\cdot x_{n})^2\, d\nu_{x_{n}}(\dot s)\leq\frac{\nu_{x_n}(q_{x_n}(E_n))}{\|f_z\|^2_{2,z}}
\leq\frac{M\nu_z(q_z(\phi_z^{-1}(V)))}{\|g\|^2_{2,z}}.
\end{equation}
An application of Fubini's Theorem gives
\begin{align}
\tr(&\Ind( x_{n},1)(D^**D))\notag\\
&= \frac{1}{4}\int_{G/S_{x_{n}}}\int_{G/S_{x_{n}}} F(s\cdot x_{n})^2 F(u\cdot x_{n})^2\Upsilon_n(s,u)^2\, d\nu_{x_{n}}(\dot s)\, d\nu_{x_{n}}(\dot u)\label{eq-forMthm}\\
&\leq \int_{G/S_{x_{n}}}\int_{G/S_{x_{n}}} F(s\cdot x_{n})^2 F(u\cdot x_{n})^2 (1+\epsilon)^2\, d\nu_{x_{n}}(\dot s)\, d\nu_{x_{n}}(\dot u)\quad\quad\quad\text{(using \eqref{eq-upsilon})}\notag
\\
&=(1+\epsilon)^2 \Big(\int_G F(s\cdot x_{n})^2\, d\nu_{x_{n}}(\dot s)\Big)^2\notag\\
&\leq\frac{(1+\epsilon)^2M^2\nu_z(q_z(\phi_z^{-1}(V)))^2}{\|g\|_{2,z}^4}\quad\quad\quad\text{(using \eqref{eq-moved})}\notag\\
&<M^2(1+\epsilon)^4 \quad\quad\quad\quad\text{(using \eqref{eq-estimate}).}\notag
\end{align}
Finally,
\begin{align*}
M_L(\Ind(z,1),(\Ind({x_n},1)))&\leq \liminf_n \tr(\Ind(x_{n},1)(D^**D))
&\leq
M^2(1+\epsilon)^4<\lfloor M^2\rfloor+1,
\end{align*}
and hence $M_L(\Ind(z,1),(\Ind(x_n,1)))\leq\lfloor M^2\rfloor$.
\end{proof}

\begin{thm}\label{thm-M}
Suppose that $(G,X)$ is a second-countable transformation group
with $G$ abelian. Let $z\in X$  and let
$(x_n)_{n\geq 1}$ be a sequence in $X$ such that $S_{x_n}\to S_z$. Assume that $G\cdot z$ is
locally closed in $X$ and that $S_z$ is compact.   Let $M\in\R$ with $M\geq 1$, and suppose that there exists an open neighbourhood $V$ of $z$ in $X$ such that $\phi_z^{-1}(V)$ is relatively compact and
\begin{equation}\label{eq-data2}
\nu_{x_n}(q_{x_n}(\phi_{x_n}^{-1}(V)))\leq M\nu_z(q_z((\phi_z^{-1}(V))))
\end{equation}
frequently. Then $M_L(\Ind(z,\tau),(\Ind(x_n,\tau_n)))\leq\lfloor M\rfloor $ for any $\tau,\tau_n\in\hat G$.
\end{thm}

\begin{proof}
As in the proof of Theorem~\ref{thm-Msquared}, it suffices to prove $M_L(\Ind(z,1),(\Ind(x_n,1)))\leq\lfloor M\rfloor$ where $x_n\to z$ and \eqref{eq-data2} holds for all $n$. Next, we claim that we may as well assume that $G\cdot z$ is the unique limit of $(G\cdot x_n)$ in $X/G$.

Since $M_L(\Ind(z,1),(\Ind(x_n,1)))\leq \lfloor M^2\rfloor<\infty$
by Theorem~\ref{thm-Msquared}, $\{\Ind(z,1)\}$ is open in
the set of limits of $(\Ind({x_n},1))_n$ \cite[Proposition~3.4]{AaH}. So there is an open
neighbourhood $U$ of $\Ind(z,1)$ in the spectrum of $C_0(X)\rtimes G$
such that $\Ind(z,1)$ is the unique limit of $(\Ind({x_n},1))_n$ in $U$.  The map $\Ind:X\to (C_0(X)\rtimes G)^\wedge$, $x\mapsto\Ind(x,1)$ is continuous  and factors through the $T_0$-isation of $X/G$ \cite[Lemma~4.10]{W2}. Hence $Y:=\Ind^{-1}(U)$ is an open $G$-invariant neighbourhood of $z$ in $X$.  Note that $x_n\in Y$ eventually, and that $x\in Y$ implies  $\phi_{x}^{-1}(V)=\phi_{x}^{-1}(V\cap Y)$.  Now we argue as in \cite[Proof of Theorem~3.5]{AaH} that if   $G\cdot x_n\to G\cdot y$ for some $y\in Y$, then $G\cdot y=G\cdot z$ since $G\cdot z$ is locally closed. Since $C_0(Y)\rtimes G$ is an ideal in $C_0(X)\rtimes G$, we may compute multiplicities in the ideal instead \cite[Proposition~5.3]{ASS}, so we may replace $X$ by $Y$  and  assume that $G\cdot z$ is the unique limit of $G\cdot x_n$ in $X/G$, as claimed.

Fix
$\epsilon>0$ such that $M(1+\epsilon)^4<\lfloor M\rfloor +1$ and
choose $\gamma>0$ such that
\begin{equation}\label{eq-Mgamma}
\gamma<\frac{\epsilon\nu_z(q_z(\phi_z^{-1}(V)))}{1+\epsilon}<\nu_z(q_z((\phi_z^{-1}(V))).
\end{equation}
It follows from the regularity of the measure $\nu_z$, as in \cite[Proof of Lemma~3.3]{AaH}, that there exists an open  relatively compact
neighbourhood $V_1$ of $z$ such that $\overline{V_1}\subset V$ and
\begin{align*}
0<\nu_z(q_z(\phi_z^{-1}(V)))-\gamma
&<\nu_z(q_z(\phi_z^{-1}(V_1)))
\leq\nu_z(q_z((\phi_z^{-1}(\overline{V_1})))\\
&\leq\nu_z(q_z(\phi_z^{-1}(V)))
<\nu_z(q_z(\phi_z^{-1}(V_1)))+\gamma.
\end{align*}
(The reason for passing from $V$ to $V_1$ is that we will  later
apply \cite[Lemma~3.2]{AaH} to the compact neighbourhood
$\overline{V_1}$ and, in contrast to what could happen with
$\overline{V}$, we can control $\nu_z(q_z(\phi_z^{-1}(\overline{V_1})))$
relative to $\nu_z(q_z(\phi_z^{-1}(V_1)))$.)

Recall that we are assuming that \eqref{eq-data2} holds for all $n$.  Thus
 \begin{align}
\nu_{x_n}(q_{x_n}(\phi_{x_n}^{-1}(V_1)))
&\leq\nu_{x_n}(q_{x_n}(\phi_{x_{n}}^{-1}(V)))\notag\\
&\leq M\nu_z(q_z(\phi_z^{-1}(V)))\notag\quad\quad\text{(by  \eqref{eq-data2})}\\
&<M\big( \nu_z(q_z(\phi_z^{-1}(V_1))) +\gamma \big)\notag\\
&<M\nu_z(q_z(\phi_z^{-1}(V_1))) +M\epsilon\big( \nu_z(q_z(\phi_z^{-1}(V)))-\gamma \big)\notag\quad\quad\text{(by \eqref{eq-Mgamma})}\\
&<M\nu_z(q_z(\phi_z^{-1}(V_1))) +M\epsilon\nu_z(q_z(\phi_z^{-1}(V_1)))\notag\\
&=M(1+\epsilon)\nu_z(q_z((\phi_z^{-1}(V_1)))\label{eq-smallV}
\end{align}
for all $n$.
Since
\[
\frac{\nu_z(q_z(\phi_z^{-1}(V_1)))\big(\nu_z(q_z((\phi_z^{-1}(V_1)))+\gamma+\frac{1}{j}  \big)}
{\big( \nu_z(q_z(\phi_z^{-1}(V_1)))-\frac{1}{j} \big)^2}
\to 1+\frac{\gamma}{\nu_z(q_z(\phi_z^{-1}(V_1)))}<1+\epsilon\]
as $j\to\infty$, there exists
$\delta>0$ such that $\delta<\nu_z(q_z(\phi_z^{-1}(V_1)))$ and
\begin{equation}\label{eq-delta}
\frac{\nu_z(q_z(\phi_z^{-1}(V_1)))\big(\nu_z(q_z(\phi_z^{-1}(\overline{V_1})))+\delta)
\big)}{\big( \nu_z(q_z(\phi_z^{-1}(V_1)))-\delta \big)^2}
<\frac{\nu_z(q_z(\phi_z^{-1}(V_1)))\big(\nu(\phi_z^{-1}(V_1))+\gamma+\delta)
\big)}{\big( \nu_z(q_z(\phi_z^{-1}(V_1)))-\delta \big)^2}
<1+\epsilon.
\end{equation}

Next we construct a function $F\in C_c(X)$ with support contained in
$V_1$. By the regularity of the measure $\nu_z$ there exists a
compact subset $W$ of the open set $q_z(\phi_z^{-1}(V_1))$ such that
$0<\nu_z(q_z(\phi_z^{-1}(V_1)))-\delta<\nu_z(W)$. Since $W$ is compact, there
is a compact neighbourhood $W_1$ of $W$ contained in
$q_z(\phi_z^{-1}(V_1))$ and a continuous function $g:G/S_z\to[0,1]$ such that
$g$ is identically one on $W$ and is identically zero off the
interior of $W_1$. Then
\begin{equation}\label{eq-g}
\nu_z(q_z(\phi_z^{-1}(V_1)))-\delta<\nu_z(W)\leq \int_{G/S_z} g(\dot t)^2\, d\nu_z(\dot t)=\|g\|_{2,z}^2.
\end{equation}

We now construct $g_1$, $g_2$, $P$, $h$, $f$ and $F$ as in the proof of Proposition~\ref{thm-Msquared}, but we note that this time $\supp f\subset \supp h\subset P
\subset V_1$.  So if $F_x(s)=F(s\cdot x)\neq 0$ then $s\in\phi_x^{-1}(V_1)$. As before
\begin{equation}\label{eq-reciprocal2}
\|\tilde f_z\|^2_{2,z} \geq\|g\|^2_{2,z}
\end{equation}
and
$\|\tilde F_z\|_{2,z}=1$.

Let $K$ be an open relatively compact symmetric neighbourhood of
$(\supp F_z)S_z(\supp F_z)^{-1}$ in $G$ and $L$ an open relatively
compact neighbourhood of $\supp F$ in $X$. Choose $b\in C_c(G\times
X)$ such that $0\leq b\leq 1/\alpha_z(S_z)$, $b$ is identically $1/\alpha_z(S_z)$ on the set
$(\supp F_z)S_z(\supp F_z)^{-1}\times\supp F$ and $b$ is identically
zero off $K\times L$. (Thus $b$ is as in Theorem~\ref{thm-Msquared},
but we have rounded it off with an open set.) Set
\[
B(r,x)=F(x)F(r^{-1}\cdot x)b(r^{-1},x)\quad
\text{and}\quad D=\frac{1}{2}(B+B^*).
\]

Again,  $\Ind(z,1)(D)$, and hence $\Ind(z,1)(D^**D)$, is
the rank-one projection determined by the unit vector $\tilde F_z\in
L^2(G/S_z,\nu)$. From \eqref{eq-forMthm} there exists $n_0$ such that
\[
\tr(\Ind(x_{n},1)(D^**D))
=\frac{1}{4}\int_{G/S_{x_n}} F(s\cdot x_{n})^2\Big(  \int_{G/S_{x_n}} F(u\cdot
x_{n})^2\Upsilon_n(s,u)^2 d\nu_z(\dot u)\Big)\, d\nu_z\nu(\dot s)
\]
where
\[
\Upsilon_n(s,u):=\int_{S_{x_{n}}}(b(us^{-1}t,s\cdot x_{n})+b(su^{-1}t,
u\cdot x_{n}))\, d\alpha_{x_{n}}(t)\leq 2(1+\epsilon)
\]
when $n\geq n_0$.
For  fixed $s\in G$,  $F(u\cdot x_n)\Upsilon_n(s,u)\neq 0$ implies that $u\in\phi_{x_n}^{-1}(V_1)$ and $us^{-1}t=uts^{-1}\in K$ for some $t\in S_{x_n}$ because $b$ is identically zero  off $K\times L$. So $u\in \phi_{x_n}^{-1}(V_1)\cap KsS_{x_n}$. Thus if $n\geq n_0$,
\begin{align*}
\tr(\Ind&(x_{n},1)(D^**D))\\
&\leq (1+\epsilon)^2\int_{s\in q_{x_n}(\phi_{x_{n}}^{-1}(V_1))}F(s\cdot x_{n})^2\Bigg( \int_{u\in q_{x_n}(\phi_{x_{n}}^{-1}(V_1)\cap KsS_{x_n})} F(u\cdot x_{n})^2\, d\nu_{x_n}(\dot u)\Bigg)\, d\nu_{x_n}(\dot s)\\
&\leq \frac{(1+\epsilon)^2}{\|\tilde f_z\|_{2,z}^4}\int_{s\in q_{x_n}(\phi_{x_{n}}^{-1}(V_1))}1\Bigg( \int_{u\in q_{x_n}(\phi_{x_{n}}^{-1}(V_1)\cap KsS_{x_n})}1\, d\nu_{x_n}(\dot u)\Bigg)\, d\nu_{x_n}(\dot s).
\end{align*}

By the regularity of $\nu_z$ there exists an open neighbourhood $U$ of
$q_z(\phi_z^{-1}(\overline{V_1}))$ such that
$\nu_z(U)<\nu_z(q_z(\phi_z^{-1}(\overline{V_1})))+\delta/2$.  Since $q_z(\phi_z^{-1}(\overline{V_1}))$ is compact, it has  an open relatively compact
neighbourhood $A$ such that $\overline{A}\subset U$.
By
\cite[Lemma~3.2]{AaH}, applied with $\overline{V_1}$, $\overline{K}$ and $q_z^{-1}(A)$,
there exists $n_1>n_0$ such that, for every $n \geq
n_1$ and every $s\in \phi_{x_{n}}^{-1}(\overline{V_1})$ there exists
$r\in\phi_z^{-1}(\overline{V_1})$ with
$\overline{K}s\cap\phi_{x_{n}}^{-1}(\overline{V_1})\subset q_z^{-1}(A)r^{-1}s$. Hence $\overline{K}sS_{x_n}\cap\phi_{x_{n}}^{-1}(\overline{V_1})\subset q_z^{-1}(A)r^{-1}sS_{x_n}$.
Since  $\nu_{x_n}$ is right-invariant we have
\[
\nu_{x_n}\big(q_{x_n}( \overline{K}sS_{x_n}\cap\phi_{x_{n_i}}^{-1}( \overline{V_1}))
\big)
\leq\nu_{x_n}(q_{x_n}(q_z^{-1}(A)S_{x_n}))=\nu_{x_n}(q_{x_n}(q_z^{-1}(A)))
\leq\nu_{x_n}(q_{x_n}(q_z^{-1}(\overline{A}))).
\]
Since $\overline{A}$ and $S_z$ are both compact, $q_z^{-1}(\overline{A})$ is compact, and
by Lemma~\ref{lem-supmeasures2} there exists $n_2>n_1$ such that $n\geq n_2$ implies  that
\[
\nu_{x_n}(q_{x_n}(q_z^{-1}(\overline{A})))
\leq \nu_z(q_z(q_z^{-1}(\overline{A})))+\delta/2 <\nu_z(q_z(\phi_z^{-1}(\overline{V_1})))+\delta.
\]
So, provided $n\geq n_2$,
\begin{align*}
\tr(\Ind(x_{n},1)(D^**D))
&\leq \frac{(1+\epsilon)^2\nu_{x_n}(q_{x_n}(\phi_{x_{n}}^{-1}(V_1)))\big(\nu_z(q_z(\phi_z^{-1}(\overline{V_1})))+\delta \big)}{\|\tilde f_z\|_{2,z}^4}\\
&<\frac{M(1+\epsilon)^3\nu_z(q_z(\phi_z^{-1}(V_1)))\big(\nu_z(q_z(\phi_z^{-1}(\overline{V_1})))+\delta
\big)}{\|\tilde f_z\|_{2,z}^4}
\text{\  using \eqref{eq-smallV}}\\
&<\frac{M(1+\epsilon)^3\nu_z(q_z(\phi_z^{-1}(V_1)))\big(\nu_z(q_z(\phi_z^{-1}(\overline{V_1})))+\delta
\big)}{\|g\|_{2,z}^4}
\text{\  using \eqref{eq-reciprocal2}}\\
&\leq\frac{M(1+\epsilon)^3\nu_z(q_z(\phi_z^{-1}(V_1)))\big(\nu_z(q_z(\phi_z^{-1}(\overline{V_1})))+\delta
\big)}{(\nu_z(q_z(\phi_z^{-1}(V_1)))-\delta)^2}
\text{\ using \eqref{eq-g}}\\
&<M(1+\epsilon)^4\text{\qquad\qquad\qquad\qquad\qquad\qquad\qquad using \eqref{eq-delta}}.
\end{align*}
By generalised lower semi-continuity \cite[Theorem~4.3]{AS}
\begin{align*}
\liminf_n \tr(\Ind(x_n,1)(D^**D))&\geq M_L(\Ind(z,1),(\Ind({x_n},1)))\tr(\Ind(z,1)(D^**D))
\\&=M_L(\Ind(z,1),(\Ind({x_n},1))).
\end{align*}
We now have
\begin{equation*}
M_L(\Ind(z,1),(\Ind({x_n},1)))\leq \liminf_n \tr(\Ind({x_n},1)(D^**D))\leq M(1+\epsilon)^4<\lfloor M\rfloor+1,
\end{equation*} and hence $M_L(\Ind(z,1),(\Ind({x_n},1)))\leq\lfloor M\rfloor$.
\end{proof}

\section{The main theorem}\label{main}

In this section we combine the results from \S\S\ref{sec-lowerbounds}--\ref{measure} to obtain our main theorem.

\begin{thm}\label{thm-main}
Suppose that $(G,X)$ is a  second-countable transformation
group with $G$ abelian. Let $z\in X$  and let
$(x_n)_{n\geq 1}$ be a sequence in $X$ such that $S_{x_n}\to S_z$ in $\Sigma$. Assume that $G\cdot z$ is
locally closed in $X$ and that $S_z$ is compact. Let $k\in\P$.  Then the following are equivalent:
\begin{enumerate}
\item\label{thm-main1} the sequence $(x_n)_n$ converges $k$-times in $X/G$ to $z$;
\item\label{thm-main2} $M_L(\Ind(z,1),(\Ind({x_n},1)))\geq k$;
\item\label{thm-main3} there exist $\tau_n,\tau\in\hat G$ such that $M_L(\Ind(z,\tau),(\Ind({x_n},\tau_n)))\geq k$;
\item\label{thm-main4} there exists $\tau_n,\tau\in\hat G$ such that $(\Ind({x_n},\tau_n))\to \Ind(z,\tau)$, and whenever $\sigma_n,\sigma\in\hat G$ such that $(\Ind({x_n},\sigma_n))\to \Ind(z,\sigma)$, $M_L(\Ind(z,\sigma),(\Ind({x_n},\sigma_n)))\geq k$;
\item\label{thm-main5} for every open  neighbourhood $V$ of $z$  such that
$\phi_z^{-1}(V)$  is relatively compact we have
\[
\liminf_n\nu_{x_n}(q_{x_n}(\phi_{x_n}^{-1}(V)))\geq k\nu_z(q_z(\phi_{x}^{-1}(V)));
\]
\item\label{thm-main6} there exists a real number $R>k-1$ such that for every open
neighbourhood $V$ of $z$  with $\phi_{z}^{-1}(V)$ relatively
compact we have
\[
\liminf_n\nu_{x_n}(q_{x_n}(\phi_{x_n}^{-1}(V)))\geq R\nu_z(q_z(\phi_{x}^{-1}(V)));\]
\item\label{thm-main7} there exists a decreasing sequence of basic compact
neighbourhoods $(W_m)_{m\geq 1}$ of $z$ such that, for each
$m\geq1$,
\[
\liminf_n\nu_{x_n}(q_{x_n}(\phi_{x_n}^{-1}(W_m)))> (k-1)\nu_z(q_z(\phi_{x}^{-1}(W_m))).\]
\end{enumerate}
\end{thm}

We show that \eqref{thm-main1} $\Longrightarrow$ \eqref{thm-main2} $\Longrightarrow$ \eqref{thm-main3} $\Longrightarrow$ \eqref{thm-main4} $\Longrightarrow$ \eqref{thm-main2} $\Longrightarrow$ \eqref{thm-main5} $\Longrightarrow$ \eqref{thm-main6} $\Longrightarrow$ \eqref{thm-main7} $\Longrightarrow$ \eqref{thm-main1}.
The reason for going  from \eqref{thm-main4} to \eqref{thm-main5} via \eqref{thm-main2} is that \eqref{thm-main2}--\eqref{thm-main4} are very similar and \eqref{thm-main2} is the least complicated to work with.

\begin{proof} \eqref{thm-main1} $\Longrightarrow$ \eqref{thm-main2}.  Assume the sequence $(x_n)_n$ converges $k$-times in $X/G$ to $z$.  Then $G\cdot x_n\to G\cdot z$, and hence $\Ind(x_n,1)\to \Ind(z,1)$.  Now $M_L(\Ind(z,1),(\Ind({x_n},1)))\geq k$ by Corollary~\ref{cor-lowermult}.

\smallskip

\eqref{thm-main2} $\Longrightarrow$ \eqref{thm-main3}. Take $\tau_n=\tau=1$.

\smallskip

\eqref{thm-main3} $\Longrightarrow$ \eqref{thm-main4}. Assume \eqref{thm-main3}. Since $M_L(\Ind(z,\tau),(\Ind({x_n},\tau_n)))>0$, $(\Ind({x_n},\tau_n))\to \Ind(z,\tau)$. Suppose $\sigma_n,\sigma\in\hat G$ such that  $(\Ind({x_n},\sigma_n))\to \Ind(z,\sigma)$.  By two applications of Proposition~\ref{prop-constant} we have
\begin{align*}
M_L(\Ind(z,\sigma),(\Ind({x_n},\sigma_n)))&=M_L(\Ind(z,1),(\Ind({x_n},1)))\\
&=M_L(\Ind(z,\tau),(\Ind({x_n},\tau_n)))\geq k,\end{align*}
so \eqref{thm-main4} holds.

\smallskip

\eqref{thm-main4} $\Longrightarrow$ \eqref{thm-main2}.  Assume \eqref{thm-main4}. Let $\tau,\tau_n\in \hat G$ such that  $(\Ind({x_n},\tau_n))\to \Ind(z,\tau)$ and $M_L(\Ind(z,\tau),(\Ind({x_n},\tau_n)))\geq k$.  Then
$M_L(\Ind(z,1),(\Ind({x_n},1)))\geq k$  by Proposition~\ref{prop-constant}, giving \eqref{thm-main2}.

\smallskip

\eqref{thm-main2} $\Longrightarrow$ \eqref{thm-main5}. Assume \eqref{thm-main2}, that is, $M_L(\Ind(z,1),(\Ind(x_n,1)))\geq k$.  Let $\epsilon>0$.  Then  $M_L(\Ind(z,1),(\Ind(x_n,1)))>\lfloor k-\epsilon\rfloor$.  By Theorem~\ref{thm-M}, for every open
neighbourhood $V$ of $z$ such that
 $\phi_z^{-1}(V)$ is relatively compact,
\[
\nu_{x_n}(q_{x_n}(\phi_{x_n}^{-1}(V)))> (k-\epsilon)\nu_zq_z(\phi_z^{-1}(V)))
\]
eventually. Thus \eqref{thm-main5} holds.
\smallskip

\eqref{thm-main5} $\Longrightarrow$ \eqref{thm-main6}. Take $R=k$.
\smallskip

For \eqref{thm-main6} $\Longrightarrow$ \eqref{thm-main7} we need the following lemma concerning accumulation of measure.

\begin{lemma}\label{lem-C} Suppose that $(G,X)$ is a transformation
group.  Let  $z\in X$ and $(x_n)_{n\geq 1}$ be a
sequence in $X$. Assume that $G\cdot z$ is locally closed in $X$ and that $S_z$ is compact.
Let $k\in \P$, and  assume that there exists a real number $R>k-1$ such that for every
open neighbourhood $U$ of $z$ with $\phi_z^{-1}(U)$ relatively
compact we have
\[
\liminf_n \nu_{x_n}(q_{x_n}(\phi_{x_n}^{-1}(U)))\geq R\nu_z(q_z(\phi_z^{-1}(U))).
\]
Then given an open  neighbourhood $V$ of $z$ such that $\phi_z^{-1}(V)$
is relatively compact, there exists a compact neighbourhood $N$ of
$z$ with $N\subset V$ such that
\[
\liminf_n \nu_{x_n}(q_{x_n}(\phi_{x_n}^{-1}(N)))> (k-1)\nu_z(q_z(\phi_z^{-1}(N))).
\]
\end{lemma}

\begin{proof} Fix $0<\gamma<\big(\frac{R-k+1}{R}\big)\nu_z(q_z(\phi_z^{-1}(V))$. By the regularity of $\nu_z$, as in \cite[Proof of Lemma~3.3]{AaH},
there exists  an open
relatively compact neighbourhood $V_1$ of $z$ with
$\overline{V_1}\subset V$ and
$
\nu_z(q_z(\phi_z^{-1}(V)))-\gamma <\nu_z(q_z(\phi_z^{-1}(V_1)))
$.
Since $\phi_z^{-1}(V_1)$ is relatively compact we have
\begin{align*}
\liminf_n\nu_z(q_{x_n}(\phi_{x_n}^{-1}(\overline{V}_1)))
&\geq \liminf_n\nu_{x_n}(q_{x_n}(\phi_{x_n}^{-1}(V_1)))\\
&\geq R\nu_z(q_z(\phi_z^{-1}(V_1))) \text{\quad by hypothesis}\\
&>R\big(\nu_z(q_z(\phi_z^{-1}(V)))-\gamma \big)\\
&>(k-1)\nu_z(q_z(\phi_z^{-1}(V)))\text{\quad  by our choice of $\gamma$}\\
&\geq (k-1)\nu_z(q_z(\phi_z^{-1}(\overline{V}_1))).
\end{align*}
So we may take $N=\overline{V}_1$.
\end{proof}

We now continue with the proof of Theorem~\ref{thm-main}

\eqref{thm-main6} $\Longrightarrow$ \eqref{thm-main7}. Assume \eqref{thm-main6} .
 Let $(V_j)_{j\geq 1}$ be a decreasing
sequence of basic  open neighbourhoods of $z$ such that
$\phi_z^{-1}(V_1)$ is relatively compact (such neighbourhoods exist
by \cite[Lemma~2.1]{AaH}). By Lemma~\ref{lem-C} there exists a
compact neighbourhood $W_1$ of $z$ such that $W_1\subset V_1$ and
\[
\liminf_n\nu_{x_n}(q_{x_n}(\phi_{x_n}^{-1}(W_1)))> (k-1)\nu_z(q_z(\phi_z^{-1}(W_1))).
\]
Now assume there are compact neighbourhoods  $W_1, W_2, \dots, W_m$
of $z$ with $W_1\supset W_2\supset\cdots \supset W_m$  such that
\begin{equation}\label{eq-3gives4}
W_i\subset V_i\text{\ and\ }\liminf_n\nu_{x_n}(q_{x_n}(\phi_{x_n}^{-1}(W_i)))>
(k-1)\nu_z(q_z((\phi_z^{-1}(W_i)))
\end{equation}
for $1\leq i\leq m$. Apply Lemma~\ref{lem-C} to  $(\interior
W_m)\cap V_{m+1}$ to obtain a compact neighbourhood $W_{m+1}$ of $z$
such that $W_{m+1}\subset (\interior W_m)\cap V_{m+1}$ and
\eqref{eq-3gives4} holds for $i=m+1$.  This gives \eqref{thm-main7}.

\smallskip

\eqref{thm-main7} $\Longrightarrow$ \eqref{thm-main1}.  Assume \eqref{thm-main7}. We show first that $G\cdot x_n\to G\cdot
z$ in $X/G$.  Let $q:X\to X/G$ be the quotient map. Let $U$ be a
neighbourhood of $G\cdot z$ in $X/G$ and $V=q^{-1}(U)$. There exists
$m$ such that $W_m\subset V$. Since
$\liminf_n\nu_{x_n}(q_{x_n}(\phi_{x_n}^{-1}(W_m)))>0$ there exists $n_0$ such that
$\phi_{x_n}^{-1}(W_m)\neq\emptyset$ for $n\geq n_0$. Thus, for
$n\geq n_0$,
\[G\cdot x_n=q(x_n)\in q(W_m)\subset q(V)= U.\]

Next, suppose that
$M_L(\Ind(z,1),(\Ind(x_n,1)))<\infty$. Then, as in the proof of
Theorem~\ref{thm-M}, we may localise to an open $G$-invariant
neighbourhood $Y$ of $z$ such that $G\cdot z$ is the unique limit
in $Y/G$ of the sequence $(G\cdot x_n)_n$.  Eventually $W_m\subset
Y$, and so the sequence $(x_n)_n$ converges $k$-times in
$Y/G$  to $z$ by Proposition~\ref{prop-tsoc2}
applied to $Y$. But now $(x_n)_n$ converges $k$-times  in $X/G$ to $z$ as well.

Finally,   suppose that
$M_L(\Ind(z,1),(\Ind(x_n,1)))=\infty$.  By Theorem~\ref{thm-M}, for every open neighbourhood $V$ of $z$ such that
$\phi^{-1}_z(V)$ is relatively compact, $\nu_{x_n}(q_{x_n}(\phi^{-1}_{x_n}(V)))\to
\infty$ as $n\to\infty$.  Let $(K_m)_{m\geq 1}$ be an increasing
sequence of compact subsets of $G$ such that $G=\cup_{m\geq
1}\interior(K_m)$ and let $(V_m)_{m\geq1}$ be a decreasing sequence
of open,  basic neighbourhoods of $z$ such  that $\phi_z^{-1}(V_1)$
is relatively compact (such neighbourhoods exist by
\cite[Lemma~2.1]{AaH}).

For fixed $m$,
\[
\nu_{x_n}(q_{x_n}(\phi_{x_n}^{-1}(V_m)))>(k-1)\nu_z(q_z(K_m)))+1>(k-1)\nu_{x_n}(q_{x_n}(K_m))
\]
eventually by Lemma~\ref{lem-supmeasures2}.
So there exists a strictly increasing sequence of positive integers $n_m$ such that, for $n\geq n_m$,
\begin{gather}\label{eq-111}
\nu_{x_n}(q_{x_n}(\phi_{x_n}^{-1}(V_m)))>(k-1)\nu_{x_n}(q_{x_n}(\phi_{x_n}^{-1}(K_m)))
\end{gather}
If $n_1>1$, then for $1\leq n<n_1$, we set $t_n^{(i)}=e$ for $1\leq i\leq k$.  For each $n\geq n_1$, there is a unique $m$ such that $n_m\leq n_{m+1}$. Choose $t_n^{(1)}\in\phi_{x_n}^{-1}(V_m)$.  Using \eqref{eq-111} we have
\begin{align*}
\nu_{x_n}(q_{x_n}(\phi_{x_n}^{-1}(V_m)\setminus K_mt_n^{(1)}S_{x_n}))
&\geq \nu_{x_n}(q_{x_n}(\phi_{x_n}^{-1}(V_m))\setminus q_{x_n}(K_mt_n^{(1)}S_{x_n}))\\
&>(k-1)\nu_{x_n}(q_{x_n}(K_m))-\nu_{x_n}(q_{x_n}(K_m))\\
&=(k-2)\nu_{x_n}(q_{x_n}(K_m)).
\end{align*}
So if $k\geq 2$ we may choose $t_n^{(2)}\in\phi_{x_n}^{-1}(V_m)\setminus  K_mt_n^{(1)}S_{x_n}$.  Continuing in this way, we obtain $t_n^{(1)}$, \dots, $t_n^{(k)}\in\phi_{x_n}^{-1}(V_m)$, such that, for $1<j\leq k$, $t_n^{(j)}\in\phi_{x_n}^{-1}(V_m)\setminus(\cup_{i=1}^{j-1}K_mt_n^{(i)}S_{x_n})$. Thus, for $n_m\leq n<n_{m+1}$ we have
\[
t_n^{(i)}\cdot x_n\in V_m\text{\ for $1\leq i\leq k$ and\ } t_n^{(j)}\notin K_mt_n^{(i)}S_{x_n}\text{\ for $1\leq i<j\leq k$}.
\]
Therefore, arguing  as in Proposition~\ref{prop-tsoc2} (with $W_m$ replaced by $V_m$) we obtain that $(x_n)_n$ converges $k$-times  in $X/G$ to $z$.
 \end{proof}



\begin{thebibliography}{99}

\bibitem{A}  R.J. Archbold,  \emph{Upper and lower multiplicity for irreducible representations of $C^*$-algebras}, Proc. London Math. Soc. \textbf{69} (1994), 121--143.


\bibitem{AD} R.J. Archbold and K. Deicke, \emph{Bounded trace $C^*$-algebras and integrable actions}, Math. Zeit. \textbf{250} (2005), 393--410.

\bibitem{AaH} R.J. Archbold and A. an Huef, \emph{Strength of convergence in
the orbit space of a transformation group}, J. Funct. Anal. \textbf{235} (2006), 90--121.

\bibitem{AaH2} R.J. Archbold and A. an Huef, \emph{Strength of convergence
and multiplicities in the spectrum of a $C^*$-dynamical system},
Proc. London Math. Soc. (3) \textbf{96} (2008) 545-581.


\bibitem{AS} R.J. Archbold and J.S. Spielberg, \emph{Upper and lower multiplicity for irreducible representations of $C^*$-algebras. \textrm{II}}, J. Operator Theory \textbf{36} (1996), 201--231.

\bibitem{ASS} R.J. Archbold, D.W.B. Somerset and J.S. Spielberg, \emph{Upper multiplicity and bounded trace ideals in $C^*$-algebras}, J. Funct. Anal. \textbf{146} (1997), 430--463.

\bibitem{DE} A. Deitmar and S. Echterhoff, Principles of Harmonic Analysis, Springer, Universitext, 2009.


\bibitem{E94} S. Echterhoff, \emph{On transformation group $C^*$-algebras with continuous trace}, Trans. Amer. Math. Soc. \textbf{343} (1994), 117--133.




\bibitem{E-memoir} S. Echterhoff, \emph{Crossed products with continuous trace}, Memoirs Amer. Math. Soc. \textbf{586} (1996).

\bibitem{ER}  S. Echterhoff and J. Rosenberg, \emph{Fine structure of the Mackey machine for actions of abelian groups with constant Mackey obstruction},
Pacific J. Math. \textbf{170} (1995), 17Ð-52.



\bibitem{fell-topology}
J.M.G. Fell, \emph{A {H}ausdorff topology for the closed subsets of a locally
  compact non-{H}ausdorff space}, Proc. Amer. Math. Soc. \textbf{13} (1962),
  472--476.


\bibitem{Gli} J.~Glimm, \emph{Locally compact transformation groups}, Trans. Amer. Math. Soc.\textbf{101} (1961), 124--138.

\bibitem{glimm-families}
J.~Glimm, \emph{Families of induced representations}, Pacific J. Math.
  \textbf{72} (1962), 885--911.

\bibitem{green1}
P. Green, \emph{{$C^*$}-algebras of transformation groups with smooth orbit
  space}, Pacific J. Math. \textbf{72} (1977), 71--97.

\bibitem{green2}
P. Green, \emph{The local structure of twisted covariance algebras}, Acta Math.
\textbf{140} (1978), 191--250.

\bibitem{aH} A.~an Huef, \emph{Integrable actions and the transformation groups
whose $C^*$-algebras have bounded trace}, Indiana Univ. Math. J. \textbf{51} (2002), 1197--1233.


\bibitem{Lud} J. Ludwig, \emph{On the behaviour of sequences in the dual of a nilpotent Lie group}, Math. Ann. 287 (1990), 239--257.

\bibitem{MW} P.S. Muhly and D.P. Williams, \emph{Continuous trace
groupoid $C^*$-algebras}, Math. Scand. \textbf{66} (1990), 231-241.

\bibitem{ped} G.K. Pedersen, Analysis Now, Springer, New York, 1989.

\bibitem{RW}
I. Raeburn and D.P. Williams,
\emph{Crossed products by actions which are locally unitary on the stabilisers},  J. Funct. Anal.  \textbf{81}  (1988),  385--431.

\bibitem{tfb}
I. Raeburn and D.P. Williams, Morita Equivalence and
Continuous-Trace $C^*$-Algebras, Math. Surveys and Monographs, vol. {\bf
60}, Amer. Math. Soc., Providence, 1998.

 \bibitem{W2}
 D.P. Williams, \emph{The topology on the primitive ideal space of
 transformation group $C^*$-algebras and CCR transformation group
 $C^*$-algebras}, Trans. Amer. Math. Soc. \textbf{266} (1981), 335--359.


\bibitem{W}  D.P. Williams, \emph{Transformation group $C^*$-algebras
with continuous trace}, J. Funct. Anal. \textbf{41} (1981), 40--76.

\bibitem{tfb^2} D.P. Williams,  Crossed Products of $C^*$-Algebras,
  Math. Surveys and Monographs, vol.~134, Amer. Math. Soc.,
  Providence, 2007.


\end{thebibliography}
\end{document}